\documentclass{article}

\usepackage{amsmath,amssymb,amsfonts}
\usepackage{ifthen}
\usepackage[amsmath,thmmarks,hyperref]{ntheorem}
\usepackage[colorinlistoftodos, textwidth=3.3cm]{todonotes}
\usepackage{natbib}
\usepackage{adjustbox}

\theoremstyle{plain}
\theoremseparator{.}
\newtheorem{theorem}{Theorem}[section]
\newtheorem{proposition}[theorem]{Proposition}
\newtheorem{lemma}[theorem]{Lemma}
\newtheorem{corollary}[theorem]{Corollary}

\theorembodyfont{}
\newtheorem{definition}[theorem]{Definition}
\newtheorem{assumption}{Assumption}

\theoremheaderfont{\itshape}
\newtheorem{remark}[theorem]{Remark}
\newtheorem{example}{Example}

\theoremsymbol{\ensuremath{\square}}
\newtheorem*{proof}{Proof}

\newcommand{\eps}{\varepsilon}
\DeclareMathOperator*{\argmin}{arg\,min}
\DeclareMathOperator{\supp}{supp}
\DeclareMathOperator{\Span}{span}
\DeclareMathOperator{\dist}{dist}
\newcommand{\field}[1]{\mathbb{#1}}
\newcommand{\R}{\field{R}}
\newcommand{\N}{\field{N}}
\newcommand{\Z}{\field{Z}}

\newcommand{\inner}[3][n]{\SwitchBracketsizeLeft{#1}\LeftBracketSize\langle#2,#3\SwitchBracketsizeRight{#1}\RightBracketSize\rangle}
\newcommand{\abs}[2][n]{\SwitchBracketsizeLeft{#1}\LeftBracketSize\lvert#2\SwitchBracketsizeRight{#1}\RightBracketSize\rvert}
\newcommand{\norm}[2][n]{\SwitchBracketsizeLeft{#1}\LeftBracketSize\lVert#2\SwitchBracketsizeRight{#1}\RightBracketSize\rVert}
\newcommand{\set}[3][a]{\SwitchBracketsizeLeft{#1}\LeftBracketSize\{#2 : #3\SwitchBracketsizeRight{#1}\RightBracketSize\}}

\newcommand{\NextScriptStyle}[1]{{\scriptstyle{#1}}}
\newcommand{\NextScriptScriptStyle}[1]{{\scriptscriptstyle{#1}}}
\newcommand{\NextTextStyle}[1]{{\textstyle{#1}}}
\newcommand{\NextDisplayStyle}[1]{{\displaystyle{#1}}}
\newcommand{\SwitchBracketsizeLeft}[1]{
  \ifthenelse{\equal{#1}{b}\OR\equal{#1}{big}}{\let\LeftBracketSize=\bigl}{
    \ifthenelse{\equal{#1}{B}\OR\equal{#1}{Big}}{\let\LeftBracketSize=\Bigl}{
      \ifthenelse{\equal{#1}{g}\OR\equal{#1}{bigg}}{\let\LeftBracketSize=\biggl}{
    \ifthenelse{\equal{#1}{G}\OR\equal{#1}{Bigg}}{\let\LeftBracketSize=\Biggl}{
      \ifthenelse{\equal{#1}{s}\OR\equal{#1}{small}}{\let\LeftBracketSize=\NextScriptStyle}{
        \ifthenelse{\equal{#1}{ss}}{\let\LeftBracketSize=\NextScriptScriptStyle}{
          \ifthenelse{\equal{#1}{t}\OR\equal{#1}{text}}{\let\LeftBracketSize=\NextTextStyle}{
        \ifthenelse{\equal{#1}{d}\OR\equal{#1}{display}}{\let\LeftBracketSize=\NextDisplayStyle}{
          \ifthenelse{\equal{#1}{a}\OR\equal{#1}{auto}}{\let\LeftBracketSize=\left}{
            \let\LeftBracketSize=\relax}}}}}}}}}}
\newcommand{\SwitchBracketsizeRight}[1]{
  \ifthenelse{\equal{#1}{b}\OR\equal{#1}{big}}{\let\RightBracketSize=\bigr}{
    \ifthenelse{\equal{#1}{B}\OR\equal{#1}{Big}}{\let\RightBracketSize=\Bigr}{
      \ifthenelse{\equal{#1}{g}\OR\equal{#1}{bigg}}{\let\RightBracketSize=\biggr}{
    \ifthenelse{\equal{#1}{G}\OR\equal{#1}{Bigg}}{\let\RightBracketSize=\Biggr}{
      \ifthenelse{\equal{#1}{s}\OR\equal{#1}{small}}{\let\RightBracketSize=\NextScriptStyle}{
        \ifthenelse{\equal{#1}{ss}}{\let\RightBracketSize=\NextScriptScriptStyle}{
          \ifthenelse{\equal{#1}{t}\OR\equal{#1}{text}}{\let\RightBracketSize=\NextTextStyle}{
        \ifthenelse{\equal{#1}{d}\OR\equal{#1}{display}}{\let\RightBracketSize=\NextDisplayStyle}{
          \ifthenelse{\equal{#1}{a}\OR\equal{#1}{auto}}{\let\RightBracketSize=\right}{
            \let\RightBracketSize=\relax}}}}}}}}}}

\newcommand{\T}{\mathbb{T}}

\DeclareMathOperator*{\esssup}{ess\,sup}

\newcommand{\BigO}[1]{\ensuremath{\operatorname{\mathcal{O}}\left(#1\right)}}
\newcommand{\E}[1]{{\mathbb E}\left[ #1 \right]}
\newcommand{\Prob}[1]{{\mathbb P}\left\{ #1 \right\}}

\allowdisplaybreaks

\begin{document}

\title{Variational Multiscale Nonparametric Regression: Smooth Functions}

\author{Markus Grasmair \smallskip \\
Department of Mathematical Sciences \\ Norwegian University of Science and Technology, 
Trondheim, Norway \\
\medskip\\
Housen Li, and Axel Munk\smallskip \\
    Institute for Mathematical Stochastics, University of G\"{o}ttingen \\
    and Max Planck Institute for Biophysical Chemistry, 
G\"{o}ttingen, Germany}

\maketitle

\begin{abstract}
For the problem of nonparametric regression of smooth functions, 
we {reconsider} and analyze a constrained variational approach, 
which we call the MultIscale Nemirovski-Dantzig (MIND) estimator.  
This can be viewed as a multiscale extension of the Dantzig selector
 (\emph{Ann. Statist.}, 35(6): 2313--51, 2009) based on early ideas of
Nemirovski (\emph{J. Comput. System Sci.}, 23:1--11, 1986). MIND
minimizes a homogeneous Sobolev norm under the constraint that 
the multiresolution norm of the residual is bounded by a universal threshold.
The main contribution of this paper is the derivation of
convergence rates of MIND with respect to $L^q$-loss, $1 \le q \le
\infty$, both almost surely and in expectation. {To this end,  we introduce} the method of
approximate source conditions. {For a
one-dimensional signal,} {these can be translated into} 
approximation properties of $B$-splines.
A remarkable consequence is that MIND attains almost
minimax optimal rates simultaneously for a large range of Sobolev and
Besov classes, {which provides certain adaptation.} 
Complimentary to the asymptotic analysis, we examine
the finite sample performance of MIND by numerical simulations.
\end{abstract}


\section{Introduction} 
In this paper, we will consider the
 nonparametric regression problem to estimate a smooth function $f\colon [0,1]^d \to \R$
from $n$ measurements
\begin{equation}\label{problem}
y_n(x) = f(x)+\xi_n(x) \quad \text{ for } x \in \Gamma_n,
\end{equation}
where $\Gamma_n$ is the regular grid on $[0,1]^d$ containing $n$
equidistant points, and $\set{\xi_n(x)}{x\in\Gamma_n}$ a set of
independent, identically distributed (i.i.d.) {centered} sub-Gaussian random
{variables} with {scale} parameter $\sigma$, i.e., the common distribution function
$\Phi$ satisfies
\begin{equation}\label{eq:sub:gauss}
\int e^{\tau t}\Phi(dt) \le e^{(\tau\sigma)^2/2} \quad
\text{ for every } \tau \in \R.
\end{equation}
For simplicity, we assume that the truth $f$ can be extended periodically {to $\R^d$} to avoid boundary effects, 
and that the noise level $\sigma$ is  known.

\subsection{Variational statistical estimation}\label{subsec:var:stat:est}

{Since the fundamental work of \citep[][]{Nad64,Sto84} and many others, 
the literature on nonparametric regression techniques has become enormously rich and diverse, and has found its way into many textbooks, see~\citep{GreSil93,FanGij96,GyoEtal02,Tsy09,KorKor11} for example.}  A prodigious amount of {these} estimation methods can be casted in a variational framework, which {can be roughly categorized into} three different formulations: \emph{penalized estimation}, \emph{smoothness-constrained estimation}, and \emph{data-fidelity-constrained estimation}, see Figure~\ref{fig:three:variational:forms}.  

\begin{figure}[!tb]
\centering
\begin{tikzpicture}
\node [draw, rectangle, rounded corners, thick] at (0, 1.25) { %
    \begin{minipage}{0.27\textwidth}
    \emph{Penalized estimation}
     \begin{equation*}
     \min_{\tilde{f}} \mathcal{L}(\tilde{f}, y_n) + \lambda \mathcal{S}(\tilde{f})
     \end{equation*}   
     \end{minipage}};
 
 \node [draw, rectangle, rounded corners, thick] at (6.3, 0) { %
    \begin{minipage}{0.36\textwidth}
    \emph{Data-fidelity-constrained estimation}
     \begin{equation*}
     \min_{\tilde{f}} \mathcal{S}(\tilde{f})\quad \text{s.t. } \mathcal{L}(\tilde{f}, y_n) \le \gamma
     \end{equation*}   
     \end{minipage}};
 
  \node [draw, rectangle, rounded corners, thick] at (6.3, 2.5) { %
    \begin{minipage}{0.36\textwidth}
    \emph{Smoothness-constrained estimation}
     \begin{equation*}
	\min_{\tilde{f}} \mathcal{L}(\tilde{f}, y_n)\quad \text{s.t. } \mathcal{S}(\tilde{f}) \le \eta
     \end{equation*}   
     \end{minipage}};
 
\draw[<->, very thick] (6.5,0.9) -- (6.5,1.6);
\draw[<->, very thick] (2.35,1.2) -- (3.2,0);
\draw[<->, very thick] (2.35,1.3) -- (3.2,2.5);
  
\end{tikzpicture}

%
%
%
%
\caption{Variational statistical estimation. 
\label{fig:three:variational:forms}}
\end{figure}
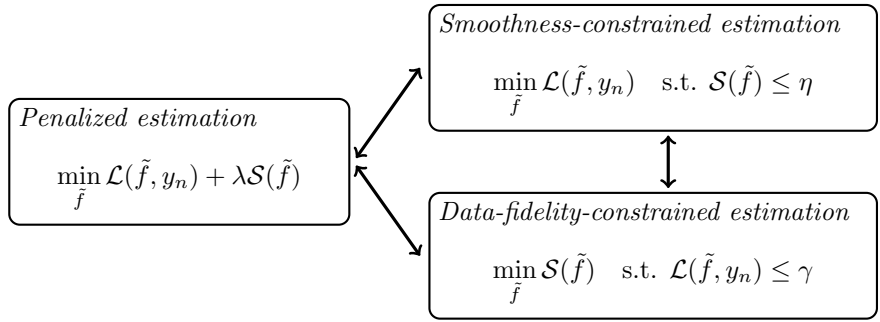

\textbf{Penalized estimation} is a solution of the Lagrangian variational problem (also known as generalized Tikhonov regularization)
\begin{equation}\label{eq:penalized:estimation}
\min_{{f}} \mathcal{L}({f}, y_n) + \lambda \mathcal{S}({f}). 
\end{equation}
The \emph{regularization term} $\mathcal{S}({f})$ accounts for {a-priori assumptions} of the truth $f$, such as smoothness, sparsity, etc.
The \emph{data fidelity term} $\mathcal{L}({f}, y_n)$ measures the deviation from the data $y_n$.  If $\mathcal{L}(\cdot, y_n)$ is the log-likelihood function of the model, {this amounts to} penalized maximum-likelihood regression~\citep[see e.g.][for general exposition]{vandeGeer88,EggLar09}. Prominent examples include smoothing splines~\citep{wahba1990spline}, {local polynomial estimators~\citep{FanGij96}}, and locally adaptive splines~\citep{MamGee97}. {It is known} that the choice of the balancing parameter $\lambda$ is in general subtle, although there are {nowadays} many data driven strategies, such as (generalized) cross validation~\citep{Wah77}, or Lepski{\u\i}'s balancing principle~\citep{Lep90}, {to mention a few. The latter} even provides adaptation over a range of generalized Sobolev scales, see e.g. \citep{GolNem97,LepMamSpo97}.

\textbf{Smoothness-constrained estimation} is to minimize the data fidelity term $\mathcal{L}$ under {the regularization constraint $\mathcal{S}$,} 
\begin{equation}\label{eq:smooth:constraint}
\min_{{f}} \mathcal{L}({f}, y_n)\qquad \text{subject to } \mathcal{S}({f}) \le \eta.
\end{equation}
It includes the well-known lasso~\citep{Tib96} {for $\mathcal{L} = \norm{\cdot - y_n}_{\ell^2}$ and $\mathcal{S} =\norm{\cdot}_{\ell^1}$} as a special case.  Another example is Nemirovski's (1985)
estimator $\hat{f}_{p,\eta}$ defined by
\begin{equation}\label{eq:def:nemirovski:est}
\hat{f}_{p,\eta} \in \argmin_{{f}}\norm{S_n{f} - y_n}_{\mathcal{B}}\qquad\text{ subject to } \norm{D^k {f}}_{L^p} \le \eta,
\end{equation} 
where $S_n$ denotes the \emph{sampling operator} on the grid $\Gamma_n$, and the \emph{multiresolution norm} $\norm{\cdot}_{\mathcal{B}}$ measures the maximum of normalized local averages (see Section~\ref{sec:mr:norm} for a formal definition). The estimator $\hat{f}_{p,\eta}$ is known to be minimax optimal (up to at most a log-factor) over {Sobolev ellipsoids $\{f; \norm{D^k f}_{L^p} \le \eta \}\subset W^{k, p}$,}
see~\citep{Nemirovski1985, NemLec00}. This indicates one drawback of this type {of estimator:} the choice of the threshold $\eta$ {determines} {a priori} the smoothness information (measured by $\mathcal{S}$) of the truth $f$, which is often unavailable in reality.  

\textbf{Data-fidelity-constrained estimation} {results from the ``reverse'' formulation} of~\eqref{eq:smooth:constraint}, given by
\begin{equation}\label{eq:data:fidelity:constraint}
\min_{{f}} \mathcal{S}({f})\qquad \text{subject to } \mathcal{L}({f}, y_n) \le \gamma.
\end{equation}
{Many basis (or dictionary) based} thresholding-type methods, such as soft-thresh\-olding \citep{Donoho95}, and block thresholding~\citep{HalEtal97,Cai99,Cai02,CaiZho09,CheFadSta10}, {can be written this way}. {Here $\gamma = \gamma_n$ {can be chosen as} a universal threshold, not depending on the data.} {For example, proper wavelet thresholding} provides spatial adaptivity, and is known to be minimax optimal for smooth functions, see~\citep{DoJo94,DJKP1995,DJKP96,HKPT98}, {while at the same time computationally fast as the thresholding is applied to each empirical wavelet coefficient, separately}. {Such adaptivity of wavelet based methods is also known for more general settings, such as linear inverse problems, see e.g.~\citep{Donobo95b,CaGoPiTs02,CohHofRei04,HofRei08}. 
The Dantzig selector~\citep{CanTao07} is also a particular data-fidelity-constrained estimator, originally introduced for linear models. In the nonparametric setting~\eqref{problem}, it has the form
\begin{equation}\label{eq:dantzig:selector}
\min_{{f} \in\R^{\Gamma_n}} \norm{{f}}_{\ell^1} \qquad \text{subject to }\norm{{f} - y_n}_{\ell^\infty} \le \gamma. 
\end{equation}
{Many other $\ell^1$-minimization approaches for recovering sparse signals also take the form of \eqref{eq:data:fidelity:constraint}, see~\citep{DonElaTem06,CaiWangXu10} for example.}

From a convex analysis point of view, all three estimation methods in Figure~\ref{fig:three:variational:forms} can be viewed as equivalent, as under weak assumptions~\citep{IvaVasTan02} each estimator in~\eqref{eq:penalized:estimation}, \eqref{eq:smooth:constraint}, \eqref{eq:data:fidelity:constraint} can be obtained as a solution of the other optimization problems via Fenchel duality~\citep[cf.][for this in the case of the lasso and the Dantzig selector]{BRT09}. 
The correspondence between {the parameters $\lambda,\eta,\gamma$,} however, is not given explicitly, and depends on the data $y_n$. It is exactly the lack of this explicit correspondence that makes the different {statistical} nature of these estimations. {From this perspective, the data-fidelity-constrained estimation~\eqref{eq:data:fidelity:constraint} has a certain appeal, since its threshold} parameter {can be chosen universally, i.e. only} determined by the noise characteristics {and the sample size $n$}, and {still} allows for a sound statistical interpretation. 
For instance, it can often be chosen in such a way that the truth $f$ satisfies the constraint {on the r.h.s. of~\eqref{eq:data:fidelity:constraint}} with probability at least $1-\alpha$, which immediately leads to the so called \emph{smoothness guarantee} of the estimate $\hat{f}$ in~\eqref{eq:data:fidelity:constraint}, 
\[
\inf_{f}\Prob{\mathcal{S}(\hat{f}) \le \mathcal{S}(f)} \ge 1- \alpha.
\]

\subsection{MIND estimator} 

In the literature, multiscale data-fidelity-constrained methods which do not explicitly rely on a specific basis or dictionary and hence do not allow for component or blockwise thresholding have also been around for some while. 
For example, \cite{Nemirovski1985} briefly discussed the ``reverse'' of his estimator~\eqref{eq:def:nemirovski:est} as well, which is given by
\begin{equation}\label{eq:revNem}
\min_{{f}} \norm{D^k{f}}_{L^p} \qquad \text{ subject to } \norm{S_n {f} -y_n}_{\mathcal{B}} \le \gamma.
\end{equation}
{These estimators} all combine variational minimization with so called multiscale testing statistics. Empirically, {they have been found} to perform very well and even outperform those explicit methods based on wavelets or dictionaries~\citep[cf.][]{CandesGuo02,DonHinRin11,FMM13}. In fact, the latter methods, as signal-to-noise ratio decreases, often show visually disturbing artifacts because of missing band pass information~\citep{CandesGuo02}.  
{On the other hand,} the computation of {such} multiscale data-fidelity-constrained estimators, {in general, leads to a high dimensional non-smooth convex optimization problem, remaining a burden for a long time.  However,  recently certain progress} has been made in the development of  algorithms for this type of problems~\citep[see][for example]{BecTeb09,ChPo11,FMM12b}.
{In the one dimensional case, fast algorithmic computation is sometimes feasible for specific functionals $\mathcal{S}$~\citep[e.g.][]{DavKov01,DavKovMei09,DueKov09,FMS14}.}
In contrast to these computational achievements, the underlying {statistical} theory for these methods is currently not well understood, in particular with regard to their asymptotic convergence behavior. In fact, there is only a small number of results in this direction we are aware of:  for fixed $k \in \N$ and $p \in [1, \infty]$, {and under the somewhat artificial assumption that the truth $f$ lies} in the constraint on the r.h.s. of~\eqref{eq:revNem}, \cite{Nemirovski1985} derived the convergence rate of~\eqref{eq:revNem} (i.e. $\mathcal{S}:=\norm{D^k \cdot}_{L^p}$) which coincides with the minimax rate over Sobolev ellipsoids in $W^{k, p}$ up to a log-factor.  Special cases of this result  have also appeared in~\citep{DavMei08} for $k = p = 2$, and in~\citep{DavKovMei09} for $k = 2,\, p=\infty$. 
In particular, adaptation of this type of estimators has not been provided so far, to the best of our knowledge. 
Intending to fill such gap, we focus on the ``reverse'' Nemirovski estimator~\eqref{eq:revNem} with $p = 2$, that is,
\begin{equation}\label{eq:constrainedprob}
\hat{f}_{\gamma_n} = \mathop{\arg\min}_{{{f}}} \frac{1}{2} \norm{D^k {f}}_{L^2}^2 
\qquad\qquad\text{ subject to }
\norm{S_n {f} - y_n}_{\mathcal{B}} \le \gamma_n,
\end{equation}
which we call the \emph{MultIscale Nemirovski-Dantzig estimator} (MIND). The choice of the name credits the fact that it is a particular ``reverse'' Nemirovski's estimator~\eqref{eq:revNem}, and  {the r.h.s.} is a (multiscale) extension of the Dantzig estimator~\eqref{eq:dantzig:selector}.

The main contribution of this work is trifold. First, we introduce the approximate source conditions~\citep{HofYam05,Hof06} from regularization theory and inverse problems into the statistical analysis of nonparametric regression.  By combining them with an improved interpolation inequality {of the multiresolution norm and Sobolev norms}, 
we are able to translate the statistical analysis into a deterministic approximation problem. The approximate source condition is essentially equivalent to smoothness concepts in terms of (approximate) variational inequalities \citep[cf.][]{HKPS07,SGGHL09,FleHof10} via Fenchel duality, see~\citep{Flemming2012}; {and conditions of this kind are fundamental for convergence analysis in inverse problems~\citep[see e.g.][Section 3.2]{EngHanNeu96}.}  

Second, we present both the $L^q$-risk convergence rate ($1 \le q \le \infty$) and the almost sure convergence rate for MIND, provided that an estimate of the approximate source condition is known. It is worth noting that the derivation of the $L^q$-risk convergence rate is more involved, for which one has to bound the size of MIND, when the truth does not lie in the multiscale constraint, {which notably extends \cite{Nemirovski1985}'s technique. Our analysis for such situation is built on the observation that the MIND estimator is always close to the data, which leads us to an upper bound on its $L^q$-loss in terms of  the multiresolution norm of the noise. The latter can be easily controlled because it has a sub-Gaussian tail.} 

Third, we show {a} \emph{partial adaptation} property of MIND in one dimension, in the sense that for a fixed $k \in \N$, it attains minimax optimality (up to a log-factor) simultaneously over Sobolev ellipsoids in $W^{s,p}$ and Besov ellipsoids in $B^{s,p'}_p$ for all $(s,p) \in [1, k] \times \{\infty\} \cup [k+1,2k]\times[2,\infty]$ and $p' \in [1,\infty]$. {These results explain to some extent the remarkably good multiscale reconstruction properties of MIND empirically found in various signal recovery and imaging applications, see Section~\ref{sec:num:example} and~\citep{CandesGuo02,DavKovMei09,FMM13}.} 

\subsection{Organization of the paper}
The rest of the paper is organized as follows: In
Section~\ref{sec:mr:norm}, we present the multiresolution norm
together with its deterministic and stochastic properties.
Section~\ref{sec:approx:source:condition} is devoted to
approximate source conditions {and so called distance functions, which provide methods for analyzing the $L^q$-loss ($1\le q \le \infty$) of MIND. }
Combining such general results
and an estimate of the distance functions, we obtain explicit
convergence rates for smooth functions, in the one dimensional case,
in Section~\ref{sec:conRate:1d}. These rates are further shown to be
minimax optimal up to a log-factor simultaneously over a large range
of Besov and Sobolev classes. In addition to the asymptotic results,
the finite sample behavior, as well as choices of the tuning
parameter, of MIND is examined empirically on simulated examples in
Section~\ref{sec:num:example}. The paper ends with discussions and
open questions in Section~\ref{sec:discuss}. Technical proofs are
given in the appendix.

\section{A heuristic explanation to MIND}

Before going into technical details, we illustrate the intuition behind {MIND's ability to recover features of the truth in a multiscale fashion} by a toy example. 

\begin{example}\label{toy:example}
Let us consider the estimation of a smooth function $f\colon [0,1] \to \R$
from  measurements
\[
y_i = f\bigl(\frac{i}{n}\bigr) + \eps_i \quad \text{ for } i = 0, \ldots, n-1, 
\]
with independent standard Gaussian error $\eps_i$. Assume now that we
have an estimator $\hat f \equiv \hat f_{s,t,a}$, such that
\[
\hat f_{s,t,a}\bigl(\frac{i}{n}\bigr) := f\bigl(\frac{i}{n}\bigr) +
s\varphi_a\bigl(\frac{i}{n}\bigr) + t \eps_i \qquad\text{ for } i = 0,
\ldots,n-1,
\]
where $s,t\ge0$, $a > 0$, $\varphi_a(x) := \sqrt{a}\varphi\left(a(x-1/2)\right)$ and 
\[
\varphi(x) := 
\begin{cases}
C e^{\frac{1}{x^2-1}} & \text{ if } \abs{x} < 1 \\
0 & \text{ if } \abs{x} \ge 1
\end{cases},
\] 
with the constant $C$ satisfying $\norm{\varphi}_{L^2} = 1$. That is,
the estimator $\hat f$ differs from the truth $f$ only by a
deterministic distortion $\varphi_a$ of scale $a$ and a random
perturbation $t\varepsilon$. By elementary computations one can show that
\begin{align*}
\Bigl|{t - \frac{s}{\sqrt{a}}}\Bigr|n \lesssim {\norm{\hat f -f }_{\ell^1}} &\lesssim \bigl(t+ \frac{s}{\sqrt{a}}\bigr)n, \\
 {\norm{\hat f - f}_{\ell^2}} &\sim (t + s)\sqrt{n},\\
 \norm{\hat f - f}_{\ell^\infty}& \sim t\sqrt{\log n} + s\sqrt{a},\\
\end{align*}
hold almost surely as $n \to \infty$.

These estimates indicate that the difference between $f$ and the estimator
$\hat f$ measured with respect to the $\ell^1$-norm depends on
the level of the random perturbation as well as the level and the
scale of the deterministic distortion. Moreover, both the random and
the deterministic part of the difference scale linearly with $n$,
which indicates that the $\ell^1$-norm is incapable of distinguishing
random from deterministic deviations. For the $\ell^2$-norm the
situation is similar. In contrast, in case of the $\ell^\infty$-norm, the
deterministic and the random part scale asymptotically differently,
and thus the $\ell^\infty$-norm can, in principle, distinguish between
these distortions. However, it also depends on the scale of the
deterministic distortion; if the scale of the deterministic distortion
is of order $\log n$, then again it is indistinguishable from random
noise.

Now note that one can also show that
\begin{equation}
 \frac{1}{\sqrt{2/(an)}}\biggl|{\sum_{\substack{i\in\N\\-\frac{1}{a}\le\frac{i}{n}-\frac{1}{2}\le\frac{1}{a}}}{\hat f(\frac{i}{n}) - f(\frac{i}{n})}}\biggr| \sim  t \sqrt{\log n} + s\sqrt{n},\label{eq:local:average}
\end{equation}
holds almost surely as $n \to \infty$.
Here, the deterministic and the random parts scale differently, and
the scale of the deterministic distortion does not influence the right
hand side of~\eqref{eq:local:average}.
These favorable properties are, however, based on the prior knowledge
of the support of the deterministic distortion $\varphi_a$, which
explicitly appears on the left hand side of~\eqref{eq:local:average}.
Still, it is possible to use the local averages
in~\eqref{eq:local:average} by taking the supremum over all possible
scales and locations of deterministic perturbation, which, basically,
results in the multiresolution norm. Later on we will see that this
approach results in the same asymptotic estimate
as~\eqref{eq:local:average}.
Therefore, the multiscale constraint of MIND with $\gamma_n \sim \log n$
guarantees that every feasible candidate contains no deterministic distortion, and the smoothness-enforcing 
regularization term selects the one without random distortion. 
The combination of both ensures that MIND avoids both deterministic and random distortions, thus close to the truth. 
\end{example}

\section{Multiresolution Norm}\label{sec:mr:norm}

Define a \emph{cube} $B$ to be a subset of $[0,1]^d$ of the form
$B = \prod_{i=1}^{d}[a_i,a_i+h)$, where $a_i \in \R$, $i=1,\ldots,d$,
and $0 < h \le 1$. By $\lvert B \rvert$ we denote its $d$-dimensional volume
$h^d$. As in~\eqref{problem}, let $\Gamma_n$ be the regular grid on $[0,1]^d$, 
\begin{equation}\label{eq:defGammaN}
\Gamma_n := \left\{(\frac{\tau_1}{n^{1/d}}, \ldots, \frac{\tau_d}{n^{1/d}}); \tau_i = 0,\ldots, n^{1/d}-1, \text{ for } i = 1,\ldots, d\right\}. 
\end{equation}

\begin{definition}\label{def:multiresolution:norm}
The \emph{multiresolution norm} $\norm{\cdot}_{\mathcal{B}}$ on $\R^{\Gamma_n}$ 
with respect to a {non-empty} system of cubes $\mathcal{B}$ is defined by
\begin{equation}\label{eq:def:mrnorm}
\norm{y}_{\mathcal{B}} := \sup_{B \in \mathcal{B}} \frac{1}{\sqrt{n(B)}} \Bigl|{\sum_{x \in \Gamma_n \cap B}y(x)}\Bigr|\quad \text{ for } y = \bigl(y(x)\bigr)_{x\in\Gamma_n} \in \R^{\Gamma_n}.
\end{equation}
Here 
\[
n(B):=\# \Gamma_n\cap B.
\]
Moreover, in the case of $n(B) = 0$, that is, $\Gamma_n\cap B = \emptyset$,
we set the term $\bigl\lvert\sum_{x \in \Gamma_n \cap B}y(x)\bigr\rvert/\sqrt{n(B)}$ to zero.
\end{definition}

Our main tool for the analysis of MIND will be estimates both for $\norm{\xi_n}_{\mathcal{B}}$
and for $\norm{S_n f}_{\mathcal{B}}$ that hold for sufficiently
rich systems of cubes. Here and in the following $S_n$ denotes the sampling operator on the set $\Gamma_n$ in~\eqref{eq:defGammaN}, provided that $f$ is continuous. 

\begin{definition}[\cite{Nemirovski1985}]\label{def:normal:sys}
A system $\mathcal{B}$ of distinct cubes is called \emph{normal}, if
there exists $c > 1$ such that
for every cube $B\subseteq [0,1]^d$ there exists a cube $\tilde{B} \in
\mathcal{B}$ such that $\tilde{B}\subseteq B$ and $\abs{\tilde{B}}
\geq \abs{B}/c$.
\end{definition}

In order to estimate the convergence rate of MIND for functions 
that are smoother than {imposed by the} regularization term, it is 
necessary to impose an additional regularity condition on the system
$B$.
\begin{definition}\label{def:m:regular}
  The system of cubes $\mathcal{B}$ is called \emph{regular} (or 
  \emph{$m$-regular}) for some $m\in\N$, $m \ge 2$,
  if it contains at least the \emph{$m$-partition system}, which is defined as all sets of the form
  $[\ell m^{-j},(\ell+1)m^{-j})$
  for all $\ell\in \N^d$, $j\in\N$.
\end{definition}

\begin{remark}\label{re:normSysExample}
Formally, a normal or regular system $\mathcal{B}$ is independent
of the grid $\Gamma_n$. Given a grid $\Gamma_n$, the value of the
multiresolution norm, however, depends on the intersection of
the cubes in $\mathcal{B}$ with $\Gamma_n$. In particular, it is the
\emph{(effective) cardinality} of distinct cubes of $\mathcal{B}$ on
$\Gamma_n$, that is, the number $\#\{B\cap \Gamma_n; B
\in\mathcal{B}\}$, that determines the computational complexity of the
evaluation of the multiresolution norm. In order to obtain numerically
feasible algorithms, one therefore would like to choose this effective
cardinality as small as possible while still ensuring normality or regularity of
$\mathcal{B}$.

\begin{itemize}
\item[a)]
The system of all distinct cubes is clearly normal and regular. Its corresponding multiresolution norm {also appears as a particular} scan statistics (maximum likelihood ratio statistic in {the} Gaussian setting), which examines the signal at every scale and location. This is a standard tool for detecting a deterministic signal with unknown spatial extent against a noisy background, see e.g.~\citep{GB99,SieYak00,DueSpok01}.
 However, the  cardinality of all the cubes on $\Gamma_n$ is $\BigO{n^2}$, 
making it {computationally} impractical for large scale problems. In practice, sparser normal systems, while retaining multiresolution nature, are {therefore} preferable.  Some examples are given below. 
\item[b)]
The system of cubes with dyadic edge lengths is normal, and of effective cardinality $\BigO{n\log n}$ on $\Gamma_n$. It is easy to see that this system is $2$-regular. 
\item[c)]
Sparse system with optimal detection power. 
In one dimension, {a normal system of $\BigO{n}$ intervals can be constructed from the system introduced in \citep{RivWal12}, by including some intervals of small scales (i.e. length $\le \log (n) / n$). This system is still sufficiently rich to be statistically optimal, in {the} setting of bump detection in the intensity of a Poisson process or in a density, but it is not regular.} The heuristics behind is that after considering one interval, not much is gained by looking at intervals of similar scales and similar locations {\citep[see also][]{ChaWal13}.}  For higher dimensions, such system can be constructed similarly, see~\citep{Wal10,SAC14}. 
\item[d)]
{The $m$-partition system has effective cardinality $\BigO{n}$ on $\Gamma_n$, and is normal and obviously $m$-regular.} As will be shown in
Section~\ref{sec:conRate:1d}, it is rich enough for nearly optimal
estimation of smooth functions (see Section~\ref{sec:num:example} for
its practical performance).  In particular, for $m=2$, it corresponds to the support set of the 
wavelet multiresolution scheme. 
\end{itemize}

It is clear that every regular system of cubes is necessarily
  normal. The converse, however, need not hold. That is, there exist
  normal systems of {cubes} that are not $m$-regular for any $m \in
  \N$ (cf. the second example above).

Finally, we note that the multiresolution norm is, actually, not necessarily
a norm but {always} a semi-norm. That is, it can happen that
$\norm{y}_{\mathcal{B}} = 0$ although the vector $y \in \R^{\Gamma_n}$ is different
from zero. Obviously this is the case if $B\cap \Gamma_n = \emptyset$ for
all $B \in \mathcal{B}$, in which case $\norm{\cdot}_{\mathcal{B}}$ is identically
zero. If the system $\mathcal{B}$ is normal, however, this situation cannot
occur for $n$ sufficiently large: the normality of $\mathcal{B}$ implies in
particular that $\mathcal{B}$ contains a cube of volume at least $1/c$,
which, for $n > c$ necessarily has a non-empty intersection with the grid $\Gamma_n$.
Still it is possible that $\norm{y}_{\mathcal{B}} = 0$ for some non-zero $y$.
On the other hand, if $\mathcal{B}$ is normal and $f\colon [0,1]^d\to \R$ is
continuous and non-zero, then there exists some $n_0\in\N$ such that
$\norm{S_n f}_{\mathcal{B}} \neq 0$ for all $n \ge n_0$, which means that
the multiresolution norm of the point evaluation of a continuous non-zero
function will eventually become non-zero.
For simplicity, we will {consider only systems $\mathcal{B}$, s.t. $\norm{\cdot}_{\mathcal{B}}$ is a norm}, 
which allows us later to define its dual norm.
Moreover, in the important case of the $m$-partition system,
it is easy to see that $\norm{\cdot}_{\mathcal{B}}$ is {indeed} a norm,
and that it can be bounded below by the maximum norm on $\R^{\Gamma_n}$.
\end{remark}

The main property of the multiresolution norm is that it allows to
distinguish between random noise and smooth functions, see Example~\ref{toy:example}. As the number $n$
of sampling points increases, the multiresolution norm of a smooth
function increases with a rate of $n^{1/2}$.
In contrast, the multiresolution norm of i.i.d.~Gaussian noise
can be expected to grow only with a rate of $\sqrt{\log n}$.
More precisely, the multiresolution norm has the following properties:

\begin{proposition}\label{prop:mulitresolution_probability}
Let $\theta > 0$, $\mathcal{B}$ be a system of cubes and
$\xi_n:=\set{\xi_n(x)}{x\in\Gamma_n}$ a set of i.i.d. sub-Gaussian random variables~\eqref{eq:sub:gauss} with parameter $\sigma > 0$. Then there exists a constant $C_{\theta}$ such that
\begin{align*}
&\Prob{\|\xi_n\|_{\mathcal{B}}\geq t} \leq \min\left\{1, 2 n^{2} e^{-\frac{t^2}{2\sigma^2}}\right\}, \\
&\E{\|\xi_n\|^\theta_{\mathcal{B}}} \leq C_{\theta} \left(\sigma{\sqrt{\log n}}\right)^\theta \quad \text{ for every } n > 1.
\end{align*}
\end{proposition} 
The first result follows from a simple union bound~\citep{Nemirovski1985}, and the second from the first using 
\[
\E{\norm{\xi}_{\mathcal{B}}^{\theta}} = \int_{0}^{\infty} \theta t^{\theta - 1} \Prob{\norm{\xi}_{\mathcal{B}} \ge t} dt. 
\]

The next result provides an interpolation inequality for the $L^q$-norm
of a function in terms of its multiresolution norm and the norm of its $k$-th
order derivative. 
For $k > d/2$, $k$, $d \in \N$ and $1 \le q \le \infty$, {let}
\begin{equation}\label{define_vartheta}
\vartheta=\vartheta(k, d, q) := 
\begin{cases}
\dfrac{k}{2k+d} & \text{if } q \le \dfrac{4k+2d}{d}, \\
\dfrac{k-d/2+d/q}{2k} & \text{if } q \ge \dfrac{4k +2d}{d}.
\end{cases}
\end{equation}

\begin{proposition}\label{prop:interpolation_regression}
Let  $\mathcal{B}$ be a normal system of cubes, $k, d\in\N$, $k > d/2$ and $1 \le q \le \infty$. Then there exist constants $C$ and $n_0$, both depending only on $k$ and $d$, such that for every $f$ and for $n\ge n_0,$
\begin{equation}\label{ineq:regression}
\|f\|_{L^q} \leq C
\max\left\{ \frac{ \|S_n f\|_{\mathcal{B}}^{2\vartheta} }{n^{\vartheta}}\|D^k f\|_{L^2}^{1-2\vartheta},\,\frac{\norm{S_n f}_{\mathcal{B}}}{n^{1/2}},\,\frac{\norm{D^k f}_{L^2}}{n^{\vartheta'}}\right\},
\end{equation}
where $\vartheta=\vartheta(k, d, q)$ is given by~\eqref{define_vartheta}
and $\vartheta' := 2k\vartheta/d$.
\end{proposition}
\begin{remark}
The inequality~\eqref{ineq:regression} actually holds for general $\norm{D^k f}_{L^p}$ ($1 \le p \le \infty$) in place of $\norm{D^k f}_{L^2}$, provided that $k > d/p$ or $k \ge d$ and $p = 1$, and proper choices of $\vartheta$, $\vartheta'$. This is a generalization of the interpolation inequality by~\cite{Nemirovski1985}.
The original version  holds only for normal systems with $c = 6$ (cf. Definition~\ref{def:normal:sys}) and $p > d$. In fact, one can prove the general inequality~\eqref{ineq:regression} similarly as~\citep{Nemirovski1985,NemLec00}, while replacing the use of Taylor polynomials and Vitali cover by that of averaged Taylor polynomials~\citep[][Chapter 4]{BreRid08} and Besicovitch cover~\citep{Bes45,Bes46}. Thus, the proof is omitted here. 

Note, that $k > d/p$ or $k \ge d$ and $p = 1$ is the weakest condition to ensure the continuity of $f$, which guarantees that the evaluation $S_n$ and the multiresolution norm $\norm{\cdot}_{\mathcal{B}}$ are well-defined.  In this sense, Proposition~\ref{prop:interpolation_regression} is already in its most general form. 
\end{remark}

\section{Approximate Source Conditions}\label{sec:approx:source:condition}

For technical simplicity, we will make the following assumption {(see, however, Section~\ref{sec:discuss})}. 
\begin{assumption}
Every function $f$ is defined on the $d$-dimensional torus $\T^d \sim \R^d / \Z^d$ (or equivalently \emph{periodic}), and has \emph{mean zero}, i.e. $$\int_{\T^d} f(x)\,dx = 0.$$ 
\end{assumption}
We denote by $L^p_{0}(\T^d)$ with $1 \le p \le \infty$ the
space of all $L^p$-functions with zero mean. 
We note that $L^p_0(\T^d)$ is a closed subspace of $L^p(\T^d)$
for all $p$. 
{Similarly, we will denote by $H^k_0(\T^d) := H^k(\T^d) \cap L^2_0(\T^d)$
the space of all $k$-th order Sobolev functions with zero mean.} {We define}
the \emph{homogeneous Sobolev norm} 
$\norm{D^k \cdot}_{L^2}( =: \norm{\cdot}_{H^k_0})$ as the norm in $H^k_0(\T^d)$.

In order to derive convergence rates for MIND, we will {now introduce more recent techniques}
from regularization theory and inverse problems, 
{which have not been applied in a statistical context so far, to the best of our knowledge.}
To that end we interpret the problem of nonparametric regression
as the inverse problem of solving the equation $S_n f = y_n$ for $f$,
where we see $S_n$ as a mapping from $H^k_0(\T^d)$ to $\R^{\Gamma_n}$, {see also \citep{BisHohMunRuy07}}. 
If $k > d/2$, which we always assume, it follows from the Sobolev
embedding theorem~\citep[see e.g.][Theorem 4.12]{AF03} that $H^k_0(\T^d)$ is continuously embedded
in the space of all continuous functions, which in turn implies that
the mapping $S_n$ is bounded.
Typical conditions in regularization theory that allow the derivation
of estimates of the quality of the reconstruction in dependence of
the actually realized noise level on $y_n$ are so called \emph{source
conditions}. In this setting, they would usually be formulated as
the condition that $f = S_n^* \omega$ for some \emph{source element}
$\omega\in \R^{\Gamma_n}$, where 
$S_n^*\colon \R^{\Gamma_n}\to H^k_0(\T^d)$ denotes the adjoint of
the sampling operator $S_n$ with respect to the norm on 
$H^k_0(\T^d)$~\citep[see][]{Gro84,EngHanNeu96}.

Such an assumption, however,  {is quite restrictive in this setting;}
for instance, for $d = 1$, it basically implies that the function $f$
is a spline.
Therefore we use a different, but related, approach based on
\emph{approximate source conditions}~{\citep[see][]{HofYam05,Hof06}}.
Here, the idea is to measure how well the function $f$ can be
approximated by functions of the form $S_n^*\omega$ for approximate
source elements $\omega$ of given norm $t \ge 0$; we thus obtain
a function $d\colon\R_{\ge 0} \to \R_{\ge 0}$, which measures for each 
$t\ge 0$ the distance between $f$ and the image of the ball of radius
$t$ under $S_n^*$. In~\citep{HofYam05}, this function $d$ has been
called \emph{distance function}. Its asymptotic properties, {as the deterministic 
``noise level'' goes to zero},  have been used
to obtain convergence rates for the solution of inverse problem.

In order to apply this approach to nonparametric regression using
the multiresolution norm, we have to consider two refinements.
First, we are interested in the asymptotics as $n\to \infty$,
which means that the operator $S_n$ we are considering changes
as well. Therefore, we will have to regard instead of a single
distance function a whole family of distance functions
$d_n\colon\R_{\ge 0}\to\R_{\ge 0}$, one for each possible grid size.
Second, since we are measuring the defect of the solution not
with respect to the usual Euclidean norm but rather with respect
to the multiresolution norm, we have to measure the approximation
quality of an approximate source element in terms of the dual
multiresolution norm~\citep[see][for a similar argumentation
in the case of Banach space regularization]{Hei08}. {This complicates
the theory considerably, since the (dual) multiresolution norm is neither 
uniformly smooth nor uniformly convex.}

\subsection{Distance Function for Multiscale Regression}\label{ss:distance}

We start by considering the dual $\norm{\cdot}_{\mathcal{B}^*}$ of
the multiresolution norm on $\R^{\Gamma_n}$ with respect to the set of
cubes $\mathcal{B}$. This norm is defined as
\[
\norm{\omega}_{\mathcal{B}^*} = \max\set{\sum_{x\in\Gamma_n}
  \omega(x)v(x)}{v \in \R^{\Gamma_n},\ \norm{v}_{\mathcal{B}} \le 1}.
\]
From the definition of the multiresolution norm in~\eqref{eq:def:mrnorm} it readily follows
that for proper real numbers $(c_B)_{B \in \mathcal{B}}$
\begin{equation}\label{eq:dual_multiresolution_norm}
\norm{\omega}_{\mathcal{B}^*} = \min\set[B]{\sum_{B\in\mathcal{B}}
    \abs{c_B}\sqrt{n(B)}}{\omega(x) = \sum_{B \ni x}  c_B \text{ for
      all } x \in \Gamma_n}.
\end{equation}

Next note that,
since $S_n\colon H^k_0(\T^d) \to \R^{\Gamma_n}$ is bounded linear, it has an adjoint $S_n^*\colon
\R^{\Gamma_n}\to H^k_0(\T^d)$, which is defined by the equation
\[
\sum_{x\in\Gamma_n} f(x)\omega(x) = 
\inner{f}{S_n^*\omega}_{H^k_0} = \inner{D^k f}{D^k
  S_n^*\omega}_{L^2}
= \int_{\T^d} D^k f\,D^k S_n^*\omega\,dx.
\]

\begin{definition}
  The \emph{multiscale distance function} for $f$ is defined as
   \[
  d_n(t):= \min_{\norm{\omega}_{\mathcal{B}^*} \le t}
  \norm{D^k S_n^*\omega - D^k f}_{L^2} 
  = \min_{\norm{\omega}_{\mathcal{B}^*} \le t}
  \norm{S_n^*\omega-f}_{H^k_{0}} .
  \]
\end{definition}

Thus the function $d_n(t)$ measures the distance between $f$
and the image of the ball of radius $t$ with respect to
$\norm{\cdot}_{\mathcal{B}^*}$ under the mapping $S_n^*$. Put
differently, it describes how well the function $f$ can be
approximated (with respect to the homogeneous $k$-th order Sobolev
norm, {$\norm{D^k\cdot}_{L^2} \equiv \norm{\cdot}_{H^k_0}$}) by functions in the range of $S_n^*$.

In the following we will provide some description of the mapping
$S_n^*$. To that end we denote for every $x \in \Gamma_n$ by $e_x \in
\R^{\Gamma_n}$ the standard basis vector at $x$ defined by 
\[
e_x(z) = \begin{cases}
  1 & \text{ if } z = x,\\
  0 & \text{ else.}
\end{cases}
\]
Moreover we define
\[
\varphi_x := S_n^* e_x.
\]
Then we have for every $f \in H^k_0(\T^d)$ the equality
\[
f(x) = \int_{\T^d} D^k u\,D^k \varphi_x\,dy.
\]
Now let $f \in H^k(\T^d)$ be arbitrary. Then $f-\int_{\T^d} f\,dz
\in H^k_{0}(\T^d)$ and therefore,
\[
f(x)-\int_{\T^d}f\,dz
= \int_{\T^d} D^k f\,D^k \varphi_x\,dz
= (-1)^k \int_{\T^d} f\,\Delta^k \varphi_x\,dz
= \inner{f}{(-1)^k\Delta^k \varphi_x}_{L^2}
\]
for every $f \in H^k(\T^d)$.
Since $f(x) = \inner{f}{\delta_x}$, we obtain that $\varphi_x
= S_n^*e_x$ is the unique weak solution in $H^k_0(\T^d)$ of
the equation\footnote{There is a solution of series form: 
\[
\varphi_x(z) = \sum_{\lambda \in \Z^d \setminus \{0\}}(2\pi\norm{\lambda})^{-2k} e^{2\pi i \lambda\cdot(z-x)} 
= \sum_{\lambda \in \N_0^d \setminus \{0\}}2(2\pi\norm{\lambda})^{-2k} \cos\left(2\pi \lambda\cdot(z-x)\right). 
\]}
\[
(-1)^k\Delta^k \varphi_x = \delta_x - 1.
\]

Moreover we have for general $\omega \in \R^{\Gamma_n}$, $\omega =
\sum_{x\in\Gamma_n} \omega_x e_x$, the representation
\[
S_n^*\omega = \sum_{x\in\Gamma_n} \omega_x \varphi_x.
\]
Then the definition of $d_n(t)$ implies that
\begin{equation}\label{eq:dn:basis:version}
d_n(t) = \min\set[B]{\norm{f - \sum_{x\in\Gamma_n}c_x
    \varphi_x}_{H^k_0}}{\norm{(c_x)_{x\in\Gamma_n}}_{\mathcal{B}^*} \le t}.
\end{equation}
Because of the definition of the dual multiresolution norm, we can
further rewrite this by introducing the functions
\[
\varphi_B := \sum_{x \in B\cap \Gamma_n} \varphi_x \qquad \text{ for }B \in \mathcal{B}.
\]
We then obtain the representation
\begin{equation}\label{eq:dn:frame:version}
d_n(t) = \min\set[B]{\norm{f-\sum_{B\in\mathcal{B}} c_B
    \varphi_B}_{H^k_0}}{\sum_{B \in \mathcal{B}} \abs{c_B} \sqrt{n(B)} \le t}.
\end{equation}

\subsection{Abstract Convergence Rates}

We consider now the MIND estimator $\hat{f}_{\gamma_n}$,
which is defined as the solution of the optimization problem
given in~\eqref{eq:constrainedprob}.
Our first result provides an estimate of the accuracy of MIND, measured in terms of an $L^q$-norm, under
the assumption that the multiresolution norm of the error
is bounded by $\gamma_n$. While the result is purely deterministic,
it immediately allows for the derivation of almost sure convergence
rates by adapting the parameter $\gamma_n$
to the number of measurements.

\begin{theorem}\label{th:constrainedrate}
  Let $k$, $d\in\N$, $k > d/2$ and $1\le q \le \infty$.  Assume that {$\mathcal{B}$ is normal} and the inequality
  \[
  \norm{\xi_n} = \norm{S_nf -y_n}_{\mathcal{B}} \le \gamma_n
  \]
  is satisfied, and denote by
  $\hat{f}_{\gamma_n}$ the MIND estimator~\eqref{eq:constrainedprob}.
  In addition, define
  \[
  c_n := \min_{t \ge 0}\bigl(d_n(t) + (\gamma_n t)^{1/2}\bigr).
  \]
  Then there exist constants $C > 0$ and $n_0\in\N$,
  both depending only on $k$ and $d$,  such that
  \begin{equation}\label{eq:constrained_est}
  \norm{\hat{f}_{\gamma_n}-f }_{L^q}
  \le C \max\Bigl\{
  \frac{\gamma_n^{2\vartheta}c_n^{1-2\vartheta}}{n^\vartheta},\
  \frac{\gamma_n}{n^{1/2}},\
  \frac{c_n}{n^{\vartheta'}}
  \Bigr\}
  \quad\text{ for } n\ge n_0,
  \end{equation}
 where $\vartheta=\vartheta(k, d, q)$ is given by~\eqref{define_vartheta}
and $\vartheta' := 2k\vartheta/d$.
\end{theorem}

\begin{proof}
See Appendix~\ref{app:constrainedrate}.
\end{proof}

As a direct consequence of the previous result and the fact that
the multiresolution norm of independent sub-Gaussian noise with high
probability increases at most logarithmically, we obtain an asymptotic
convergence rate almost surely for MIND with properly chosen $\gamma_n$, {i.e.}
\begin{equation}\label{eq:choice:gammaN}
\gamma_n = C {(\log n)}^{r}\quad \text{ for some } r \ge\frac{1}{2} \text{ and } C > 
\begin{cases} 
0 & \text{ if } r > \frac{1}{2},\\
\sigma\sqrt{5+\frac{2k}{d}} & \text{ if } r = \frac{1}{2}.
\end{cases}
\end{equation}
We emphasize that such choice of $\gamma_n$ is universal, in the sense that it is 
independent of the smoothness of the truth $f$, and 
the system of cubes $\mathcal{B}$. In particular, when $r > 1/2$, $\gamma_n$ depends  
on $n$ only.

\begin{corollary}\label{cor:constrainedrate}
  Assume that $\mathcal{B}$ is normal, {that $\gamma_n$ is
    chosen as in~\eqref{eq:choice:gammaN}}, and that
  \[
  \min_{t \ge 0} \bigl(
  d_n(t)+(\log n)^{r/2} t^{1/2}\bigr)
  = \BigO{n^{-\mu}}
  \]
  for some $0 \le \mu < 1/2$.
  Then there exists a constant $C$ such that the MIND estimator $\hat{f}_{\gamma_n}$ satisfies the estimate
  \begin{equation}\label{eq:constrained_est_rate}
  \mathop{\lim\sup}_{n\to \infty} \left(n^{\mu(1-2\vartheta)+\vartheta} {(\log n)}^{-2r\vartheta} \norm{\hat{f}_{\gamma_n}-f}_{L^q} \right)\le C \quad a.s.
  \end{equation}
    with $\vartheta$ given in~\eqref{define_vartheta}.
\end{corollary}

\begin{proof}
  With the given choice of $\gamma_n$,
  Proposition~\ref{prop:mulitresolution_probability} implies that
  \[
  \Prob{\norm{\xi_n}_{\mathcal{B}} > \gamma_n} \to 0
  \qquad\qquad\text{ as } n \to \infty.
  \]
  As a consequence, the probability that the estimate
  in Theorem~\ref{th:constrainedrate} applies tends to 1 as $n\to \infty$.
  Noting that, for $n$ sufficiently large and $0 \le \mu < 1/2$,
  the first term on the right hand side of~\eqref{eq:constrained_est}
  is always dominant, we obtain~\eqref{eq:constrained_est_rate}.
\end{proof}

{Moreover, we obtain under the same assumptions also the same convergence rate
in expectation. The proof of this result, however, is more involved, because
it requires an estimate for the error $\norm{\hat{f}_{\gamma_n}-f}_{L^q}$ in the
high noise case $\norm{\xi_n}_{\mathcal{B}} > \gamma_n$, in which case the
estimate from Theorem~\ref{th:constrainedrate} does not apply.
Thus it is relegated to the appendix.}

\begin{theorem}\label{th:rate_expect}
{Assume the setting of Theorem~\ref{th:constrainedrate} and Corollary~\ref{cor:constrainedrate}. }
  Then the MIND estimator $\hat{f}_{\gamma_n}$ satisfies 
  \begin{equation}\label{eq:expectedrate}
  \E{\norm{\hat{f}_{\gamma_n}-f }_{L^q}} = \mathcal{O}\left(n^{-\mu(1-2\vartheta)-\vartheta} {(\log n)}^{2r\vartheta}\right)
  \end{equation}
  as $n \to \infty$, with $\vartheta$ given in~\eqref{define_vartheta}.
\end{theorem}

\begin{proof}
  See Appendix~\ref{app:expectedrate}.
\end{proof}

\begin{remark}\label{rem:natural:bnd:dist:fun}
  We note that the inequality
  \[
  c_n = \min_{t \ge 0} \bigl(d_n(t)+(\log n)^{r/2} t^{1/2}\bigr) \le d_n(0) = \norm{D^k f}_{L^2}
  \]
  always holds. Under the assumption that 
  $$
  f \in W^{k,p}_0(\T^d) \qquad \text{ for some } p \in [2,  \infty],
  $$
  we therefore always obtain with the parameter choice given
  in Corollary~\ref{cor:constrainedrate} and Theorem~\ref{th:rate_expect},
  respectively, for MIND 
  \[
  \norm{\hat{f}_{\gamma_n}-f }_{L^q} = \mathcal{O}\left(n^{-\vartheta} (\log n)^{2r\vartheta}\right),
  \]
  almost surely and in expectation. 
\end{remark}

We conclude this section by emphasizing that the introduced multiscale distance functions $d_n(t)$ transform the convergence analysis of MIND into the study of approximation property of the bases $\bigl(\varphi_x\bigr)_{x \in \Gamma_n}$, or the frames $\bigl(\varphi_B\bigr)_{B \in \mathcal{B}}$, with the size of coefficients controlled in certain sense, see~\eqref{eq:dn:basis:version} and \eqref{eq:dn:frame:version}. In one dimension, we are able to derive sharp error bounds for such approximation problem {(see Section~\ref{sec:conRate:1d})}, using the well-developed theory of B-splines, see the next section. However, in higher dimensions, this approximation problem~\eqref{eq:dn:basis:version} or \eqref{eq:dn:frame:version} remains still open. Note that there exist some results on the approximation error of bases $\bigl(\varphi_x\bigr)_{x \in \Gamma_n}$~\citep[see][]{DNW99,NSW02,NWW03}, but we are not aware of any results about the size of the coefficients.

\section{Convergence Rates for $d=1$}\label{sec:conRate:1d}

As a first step, we show that the range of the
adjoint $S_n^*$ of the sampling operator consists basically of
splines. Moreover, it is possible to obtain estimates for the dual
multiresolution norm of splines provided that the system of intervals
on which the multiresolution norm is based is sufficiently
regular. The desired approximate source conditions follow then from
approximation results for splines.
In the following, we will introduce the necessary notation and state
our main theorems, while the major proofs are, again,
postponed to the appendix.

\subsection{Notation}\label{subsec:basics}

We denote the forward and backward differences of a function 
{$f \colon \T\to\R$} by
\[
D_{h,+} f(\cdot) = f(\cdot+h) - f(\cdot), \text{ and } D_{h,-} f(\cdot) = f(\cdot) - f(\cdot-h) \text{ with }h \in\R,
\]
and that of a sequence {$\{a_i\}_{0\le i \le n-1}$} by
\[
{(D_+ a)_i} = a_{i+1} - a_{i}, \text{ and } {(D_{-} a)_i} = a_{i} - a_{i-1}.  
\]
{Here we define $(D_+ a)_{n-1} = a_0 - a_{n-1}$ and $(D_- a)_0 = a_0 - a_{n-1}$.}
We note that {the adjoints of these mappings}
are given, respectively, by
\[
D_{h,+} ^* = - D_{h,-} \text{ and } D_+^* = - D_{-}.
\]

In the following, we will give a brief introduction to Besov spaces, 
{as it is required here,}
 and refer the interested readers to \citep{Triebel1983,Triebel1992,Triebel1995,AF03} for further details.   
First we define the $r$-th \emph{modulus of smoothness} of $f$ in $L^p(\T)$, $1\le p \le \infty$, as 
\[
\varpi_r(f;t)_p := \sup_{0 \le \abs{h} \le t}\norm{D^r_{h,+} f}_{L^p}.
\]
Based on it, we define the \emph{Besov norm} $\norm{\cdot}_{B^{s,p'}_{p,0}}$, with $s > 0, 1 \le p, p' \le \infty$, as
\[
\norm{f}_{B^{s,p'}_{p,0}} :=  \norm{f}_{L^p} + \abs{f}_{s,p,p',r}, 
\]
where $s < r$, $r \in \N$,  and 
\[
\abs{f}_{s,p,p',r}:=\begin{cases}
\left(\int_{\T} \left(t^{-s}\varpi_r(f;t)_p\right)^{p'}\frac{dt}{t}\right)^{1/p'} & \text{ if } 1\le p' < \infty\\
\esssup_{t > 0} t^{-s} \varpi_r(f;t)_p & \text{ if } p' = \infty. 
\end{cases}
\]
The \emph{Besov space} $B^{s,p'}_{p,0}(\T)$ is then defined as the Banach space consisting of functions with bounded Besov norm, that is, 
\[
B^{s,p'}_{p,0}(\T) := \{f \in L^p_{0}(T); \norm{f}_{B^{s,p'}_{p,0}} < \infty\}. 
\]
An equivalent definition of Besov spaces is based on interpolation
theory of Banach spaces, for instance using the K-method, see e.g.~\citep{Triebel1995}.
Note that the Sobolev space $W^{s,p}_{0}(\T^d) $ equals the Besov space $ B^{s, p}_{p, 0}(\T^d)$ for every non-integer $s\in(0,\infty)$, but $W^{k,p}_{0}(\T^d) \neq B^{k, p}_{p, 0}(\T^d)$ for $k\in \N$.

Given $m\in\N$, denote by $\mathcal{P}_m$ the space of polynomials of order $m$ (i.e. of degree $\le m-1$), that is, 
\[
\mathcal{P}_m := \Bigl\{\sum_{i=1}^m a_i x^{i-1} : a_i \in \R, \, i = 1,\ldots,m\Bigr\}.
\]
Now assume that $\Gamma \subset \T$ is a discrete subset.
The space of piecewise polynomials of order $m$ on $\T$ with knots in $\Gamma$ is
defined by 
\begin{multline*}
\mathcal{PP}_m(\Gamma;\T) := \Bigl\{p\colon\T\to\R :
\text{for all } (x,y) \subset \T\setminus\Gamma  \\
\text{ there exists } q \in \mathcal{P}_m \text{ s.t. }
p(t) = q(t) \text{ for all } t \in (x,y)\Bigr\}.
\end{multline*}
Then we define the space of $m$-order splines on $\T$ with simple knots in $\Gamma$ as 
\[
\mathcal{S}_m(\Gamma;\T):=\mathcal{PP}_m(\Gamma;\T)\cap \mathcal{C}^{m-2}(\T).
\] 
Let $Q^m_{0} \in \mathcal{S}_m(\Gamma_n;\T)$ be given by
\[
Q^m_{0}(x) := \frac{n^{m-1}}{(m-1)!} \sum^m_{i=0}
(-1)^i{{m}\choose{i}}\Bigl(x-\frac{i}{n}\Bigr)_{+}^{m-1}\quad \text{for } x \in
[0,1),
\]
where $(x)_+:=\max\{x,0\}$. 
Then $\set{Q^m_i(x):=Q^m_0(x-i/n)}{i=0,\ldots,n-1}$ forms a basis of
$\mathcal{S}_m(\Gamma_n;\T)$, which is called the basis of \emph{normalized
  B-splines}. More details can be found
in~\citep{wahba1990spline,Schumaker2007} for example.

\subsection{Dual operator and splines}

For each $i  = 0,1,\ldots,n-1$, we denote by $\varphi_{i,n}$
the unique weak solution of the differential equation
\begin{equation}\label{eq:Green_fun}
(-1)^k\varphi^{(2k)}_{i,n}  = \delta\Bigl(\frac{i}{n} - \cdot\Bigr) -1, \,
\varphi_{i,n}  \in H^k_{0}(\T).
\end{equation}
{As demonstrated in Section~\ref{ss:distance},}
it follows that $S^*_n e_{i/n} = \varphi_{i,n}$.
We will show in the following that {the} span of the functions $\varphi_{i,n}$
in particular contains the space of all splines of order $2k$
on $\Gamma_n$ with zero mean.

To that end,
let us first define $\chi_{n} \in L^2(\T)$ by
\[
\chi_n(y):=\begin{cases}
1, \text{ if } y \in [0,\frac{1}{n}), \\
0, \text{ if } y \in [\frac{1}{n},1).
\end{cases}
\]
By integrating both sides of~\eqref{eq:Green_fun}
and respecting the zero mean we obtain
\[
(-1)^{k-1}\varphi^{(2k-1)}_{i,n}(z) = \begin{cases}
z - \frac{i}{n}+\frac{1}{2}, &\text{ if } 0 \le z < \frac{i}{n}, \\
z - \frac{i}{n}-\frac{1}{2}, &\text{ if } \frac{i}{n} \le z < 1.
\end{cases}
\]
Therefore
\[
(-1)^{k}D_{\frac{1}{n},-}\varphi_{i,n}^{(2k-1)}(z) = \chi_n\Bigl(z-\frac{i}{n}\Bigr)-\frac{1}{n}.
\]
Repeating this procedure $m$ times (with $m\le 2k$), we see that
\[
(-1)^k D^{m}_{\frac{1}{n},-}\varphi^{(2k-m)}_{i,n}(z) = (\chi_n*^{m-1}\chi_n)\Bigl(z-\frac{i}{n}\Bigr)-\frac{1}{n^m}.
\]
As a consequence, it follows that
\begin{multline}\label{eq:def:psi}
\psi_{i,n}^m(z):=(-1)^k n^{m-1} D^{m}_{\frac{1}{n},-}\varphi^{(2k-m)}_{i,n}(z) 
= n^{m-1}(\chi_n*^{m-1}\chi_n)(z-\frac{i}{n}) - \frac{1}{n} 
= Q^m_i(z)-\frac{1}{n}
\end{multline}
is the $L^2$-projection of the normalized B-spline $Q^m_i$ onto $L^2_{0}(\T)$. 
We do note here that the functions $\psi_{i,n}^m$ are not linearly independent,
their sum being zero.

Now assume that
\[
h = \sum_{i=0}^{n-1} \tilde{c}_i \psi_{i,n}^m
\]
for some coefficients $\tilde{c}_i \in \R$.
Noting that
\[
D_{1/n,-} \varphi_{i,n}^{(k)}(z) = \varphi_{i,n}^{(k)}(z)-\varphi_{i,n}^{(k)}\Bigl(z-\frac{1}{n}\Bigr)
= \varphi_{i,n}^{(k)}(z) - \varphi_{i+1,n}^{(k)}(z),
\]
we see that
\begin{align*}
h = &(-1)^k n^{m-1} \sum_{i=0}^{n-1} \tilde{c}_i D_{\frac{1}{n},-}^m \varphi_{i,n}^{(2k-m)} \\
= &(-1)^k n^{m-1} \sum_{i=0}^{n-1} \tilde{c}_i (D_{\frac{1}{n},-}^{m-1}\varphi_{i,n}^{(2k-m)}-D_{\frac{1}{n},-}^{m-1}\varphi_{i+1,n}^{(2k-m)})\\
= &(-1)^k n^{m-1} \sum_{i=0}^{n-1} (D_{-}\tilde{c})_i D_{\frac{1}{n},-}^{m-1}\varphi_{i,n}^{(2k-m)}.
\end{align*}
Repeating this argumentation $m$ times, we obtain
\[
h = (-1)^k n^{m-1} \sum_{i=0}^{n-1} \tilde{c}_i D_{\frac{1}{n},-}^m \varphi_{i,n}^{(2k-m)}
= (-1)^k n^{m-1} \sum_{i=0}^{n-1} (D_-^m \tilde{c})_i \varphi_{i,n}^{(2k-m)}.
\]
This shows that, indeed, the span of the functions $\psi_{i,n}^m$ is contained
in the span of the functions $\varphi_{i,n}^{(2k-m)}$, and that the change
of coefficients with respect to the different spanning sets is given
by the linear mapping $\tilde{c}\mapsto (-1)^k n^{m-1} D_-^m\tilde{c}$.

\begin{remark}
It is still interesting to know that 
$$
\varphi_{i,n}(x) \equiv (-1)^{k-1}B_{2k}(x - \frac{i}{n}),
$$
with $B_{2k}$ the Bernoulli 
polynomial~\citep[see e.g.][Section 9.4]{Kress98},
\[
B_{2k}(x) := 2(-1)^{k-1}\sum_{l = 1}^{\infty} \frac{\cos (2\pi l x)}{(2 \pi l)^{2k}},
\]
although this fact is not needed in this paper.
One can easily see this by means of Fourier series. 
\end{remark}

\subsection{Convergence Rates}
We now derive the main results of this paper,
where we prove convergence rates in the one-dimensional case
for $f$ contained in various Sobolev and Besov spaces.

{Our first main result in the one-dimensional setting is concerned
with the high regularity situation, where the function $f$ actually
is of higher smoothness than assumed by the regularization term
$\norm{D^k \hat{f}}_{L^2}^2$. In this case, it turns out that indeed
a higher order convergence rate is obtained than the one discussed
in Remark~\ref{rem:natural:bnd:dist:fun}. {For this to hold, however, we have to assume that 
the system of intervals $\mathcal{B}$ is regular (see Definition~\ref{def:m:regular}), which implies its normality. }
 The proof of this result, mainly postponed to the appendix, relies on estimates for the {multiscale
distance function} $d_n$, which in turn follow from various approximation
results with splines.}

\begin{proposition}\label{pr:distance_1d}
  Assume that $d = 1$, $r \ge 1/2$, that $\mathcal{B}$ is regular, and that
  \[
  f \in B_{p,0}^{k+s,p'}(\T)
  \]
  for some $1 \le s \le k$ and $1 \le p,p' \le \infty$.
  Then
  \[
  \min_{t \ge 0} \bigl(
  d_n(t)+(\log n)^{r/2}t^{1/2}\bigr)
  = \BigO{n^{-\mu}(\log n)^{2r\mu}}
  \]
  with
  \begin{equation}\label{eq:defMU}
    \mu =  \frac{s - \Bigl(\frac{1}{p}-\frac{1}{2}\Bigr)_+}{2s+2k+1 -2 \Bigl(\frac{1}{p}-\frac{1}{2}\Bigr)_+}.
  \end{equation}
  {The same result holds for $f \in W^{k+s,p}_{0}(\T)$ with $1 \le s \le k$
  and $1 \le p \le \infty$.}

\end{proposition}

\begin{proof}
  See Appendix~\ref{app:proof_dist_1d}.
\end{proof}

\begin{theorem}\label{th:rate_higherorder}
  Assume that $d = 1$, that $\mathcal{B}$ is regular, and that
  \[
  f \in B_{p,0}^{k+s,p'}(\T)
  \]
  for some $1 \le s \le k$ and $1 \le p,p' \le \infty$.
  Then the MIND estimator $\hat{f}_{\gamma_n}$ 
  {satisfies, with a parameter choice $\gamma_n$
    given by~\eqref{eq:choice:gammaN},}
   \[
  \norm{\hat{f}_{\gamma_n}-f }_{L^q} =
  \mathcal{O}\left(n^{-\mu(1-2\vartheta)-\vartheta} {(\log
      n)}^{2r\mu(1-2\vartheta)+2r\vartheta}\right)\qquad\text{ as } n \to \infty,
  \]
  a.s. and in expectation, with $\vartheta$  in~\eqref{define_vartheta} and $\mu$ in~\eqref{eq:defMU}.  
  The same result holds for $f \in W^{k+s,p}_{0}(\T)$ with $1 \le s \le k$
  and $1 \le p \le \infty$.
\end{theorem}

\begin{proof}
  This is a direct consequence of Proposition~\ref{pr:distance_1d},
  Corollary~\ref{cor:constrainedrate}, and Theorem~\ref{th:rate_expect}
\end{proof}

\begin{remark}
  Note that the rate obtained in the previous result greatly
  simplifies in the case where $p \ge 2$ and $q \le 4k+2$. Then, a
  short computation shows that it can be written as
  \[
  {\norm{\hat{f}_{\gamma_n}-f }_{L^q}} =
  \mathcal{O}\left(n^{-\frac{k+s}{2k+2s+1}} {(\log
      n)}^{\frac{2k+2s}{2k+2s+1}r}\right)\qquad\quad\text{ as } n \to \infty.
  \]
\end{remark}

In the one-dimensional case, it is also possible to obtain convergence
rates in the case where the regularity of the function $f$ is
overestimated by the regularization term.
In this case, the approach based on the multiscale distance function does not
readily apply, because it is inherently based on the assumption that
$f \in H^k_0(\T)$. Instead, it is possible to approximate $f$
by a sufficiently regular function, for which then the higher order
results can be applied. The final convergence rate then follows from
a combination of these higher order rates and the approximation error.

\begin{theorem}[Over-smoothing]\label{th:oversmoothing}
Let  $\mathcal{B}$ be normal, $d = 1$, $k \in \N$, $1 \le q \le 4k + 2$ and 
$$
f \in W^{s,\infty}_{0}(\T)\text{ or } B^{s,p'}_{\infty,0}(\T) \qquad \text{ with } s\in[1,k]\text{ and } p'\in[1,\infty].
$$
Let also $\hat f_{\gamma_n}$  be the MIND estimator by~\eqref{eq:constrainedprob} 
with the homogeneous Sobolev norm of order $k$, $\norm{D^k\cdot}_{L^2}$,  
 and the threshold $\gamma_n$ in~\eqref{eq:choice:gammaN}. Then 
it holds that, 
\[
\norm{\hat{f}_{\gamma_n} - f}_{L^q} = \mathcal{O}\left(n^{-s/(2s+1)} (\log n)^{\epsilon+s/(2s+1)}\right)\qquad\text{ as } n \to \infty,
\]
a.s. and in expectation, 
for any $\epsilon > \frac{(2r-1)k}{2k+1}$ with $r$ in~\eqref{eq:choice:gammaN}.
\end{theorem}

\begin{proof}
See Appendix~\ref{sec:over:smooth}.
\end{proof}

\begin{remark}
  For simplicity, the convergence rates results of Theorems~\ref{th:rate_higherorder}
   and~\ref{th:oversmoothing} were only given
  in $\mathcal{O}$ notation. However, it is worth pointing out that
  the proofs, if followed closely, actually also provide the constants
  in these rates. Most importantly, one can show that the constant
  only depends on the norm of $f$ in the corresponding Besov or
  Sobolev space. More precisely, it can be shown that the constant can,
  in the Besov space case, be written in the form
  $C\norm{f}_{B_{p,0}^{k+s,p'}}^{1-2\vartheta}$ with $C > 0$
  only depending on $k$, $s$, $p$, $p'$, and $q$, and the analogous
  result holds for the Sobolev space case. {As we will see in the 
  next subsection, this observation
  leads to the partial adaptation property of the MIND estimator, in minimax sense. }
\end{remark}

\subsection{Comparison with Minimax Rates}\label{subsec:minimax:rate}

Given a class $\mathcal{F}$ of continuous functions, we define the minimax
$L^q$-risk of nonparametric regression~\eqref{problem} over $\mathcal{F}$ by
\[
\mathcal{R}_{q}(n; \mathcal{F}) := \inf\Bigl\{\sup_{f\in \mathcal{F}}\E{\norm{\hat f -
      f}_{L^q}} : \hat f \text{ is an estimator}\Bigr\}.
\]
In other words, we measure for each estimator, 
the maximal expected error over all functions $f \in \mathcal{F}$,
and then compute the infimum of this maximal error over the class of
all estimators.

In the case of $\mathcal{F}$ consisting of Sobolev or Besov functions of a
certain regularity, it is possible to derive explicit lower bounds for
the minimax risk $\mathcal{R}_q$. To that end, we introduce, for $s
\ge 0$, $1 \le p \le \infty$, and $L > 0$ the Sobolev ball
\begin{equation}\label{eq:sobolev:ball}
S^{s,p}_{L} := \left\{f\in W^{s,p}_{0}(\T) : \norm{f}_{W^{k,p}_{0}} \le L\right\},
\end{equation}
and for $s \ge 0$, $1 \le p,p' \le \infty$, and $L > 0$ the Besov ball 
\begin{equation}\label{eq:besov:ball}
B^{s,p,p'}_{L} := \left\{ f \in B^{s,p'}_{p,0}(\T) : \norm{f}_{B^{k,p'}_{p,0}} \le L\right\}.
\end{equation}
In \citep{Nemirovski1985} it has been shown that, for $s \in \N$, and $n$ sufficiently large, there exists a constant $C > 0$
depending only on $s$ such that
\begin{equation}\label{eq:minimax:rate}
\mathcal{R}_{q}(n; S^{s,p}_{L}) \ge C 
\begin{cases}
\left(\frac{\sigma^2}{n}\right)^{\beta}L^{1-2\beta} & \text{ if } q < (2s+1)p ,\\
\left(\frac{\sigma^2\log n}{n}\right)^{\beta}L^{1-2\beta} & \text{ if } q \ge (2s+1)p,\\
\end{cases}
\end{equation}
where
\[
\beta := 
\begin{cases}
\frac{s}{2s+1} & \text{ if } q < (2s+1)p, \\
\frac{s-1/p+1/q}{2s+1-2/p} & \text{ if } q \ge (2s+1)p.\\
\end{cases}
\] 
Following the proof in~\citep{Nemirovski1985}, one can show that this result
still holds for non-integer $s > 1/p$, and also for all the Besov
balls $B^{s,p,p'}_{L}$ with $s > 1/p$. Even more, in the case
of $q = (2s+1)p$, the lower bound can be tightened to include the
logarithmic factor $(\log n)^{(1-p/\min\{p,p'\})_+/q}$, see~\citep[][Theorem 1]{DJKP96}
for details.

\subsubsection{Partial adaptation}\label{subsubsec:partial:adapt}
Comparing these minimax $L^q$-risks with the convergence rates of 
MIND in Theorems~\ref{th:rate_higherorder} and~\ref{th:oversmoothing}, 
and Remark~\ref{rem:natural:bnd:dist:fun},
we see that, for $1 \le q \le 4k+2$, the polynomial part
of our rates coincides with the polynomial part of the minimax risk in
case either the function $f$ is contained in the Sobolev space
$W_0^{s,p}(\T)$ with either $1 \le s \le k$ and $p = \infty$,
$s = k$ and $2 \le p \le \infty$, or $k+1 \le s \le 2k$ and $p \ge
2$ (see Figure~\ref{fig:optrates}). In other words, in all of these
cases, the convergence rates we obtain with MIND are
optimal up to a logarithmic factor.

We want to stress here that our convergence rates do not rely on a
precise knowledge of the smoothness class of the function $f$. In
contrast, the regularization parameter $\gamma_n$ does only depend on
the sample size, and the smoothing order of the regularization term
need only be a rough guess of the actual smoothness of $f$. Neither in
the case where the smoothness of $f$ is overestimated nor in the case
where it is slightly underestimated do we obtain results that are,
asymptotically, far from being optimal. The  method MIND 
automatically adapts to the smoothness of the function $f$ independent
of our prior guess.

Note that the adaptation range of MIND scales with the smoothness order 
of regularization $k$. This suggests that we should choose $k$ as large as possible. 
The minimization problem in~\eqref{eq:constrainedprob} becomes, however,  
more numerically unstable as $k$ increases. Thus, the choice of $k$ should balance 
the performance and the numerical stability. In practice, we found that it works fine for
$k = 1,2, 3$ (cf. Section~\ref{subsec:choice:k}).

\begin{figure}[!htb]
\centering
\begin{tikzpicture}[scale=.65]

\draw[->] (1,1) -- (1,8);
\draw (0.7,8.3) node {$p$};
\draw[->] (1,1) -- (8,1);
\draw (8.3,0.7) node {$s$};

\draw (2,1) -- (2,1.1);
\draw (2,1) node[anchor=north] {$1$};
\draw (4,1) -- (4,1.1);
\draw (4,1) node[anchor=north] {$k$};
\draw (5,1) -- (5,1.1);
\draw (5,1) node[anchor=north] {$k+1$};
\draw (7,1) -- (7,1.1);
\draw (7,1) node[anchor=north] {$2k$};
\draw (1,1) node[anchor=north] {$0$};

\draw (1,7) -- (1.1,7);
\draw (1,7) node[anchor=east] {$\infty$};
\draw (1,3) -- (1.1,3);
\draw (1,3) node[anchor=east] {$2$};
\draw (1,2) -- (1.1,2);
\draw (1,2) node[anchor=east] {$1$};
\draw (1,1) node[anchor=east] {$0$};

\draw[color=gray,dotted] (1,7)--(7,7);
\draw[color=gray,dotted] (1,3)--(7,3);
\draw[color=gray,dotted] (5,1)--(5,7);
\draw[color=gray,dotted] (7,1)--(7,7);
\filldraw[fill=red!30,draw=red,very thick] (5,3) rectangle (7,7);

\draw[color=gray,dotted] (2,1)--(2,7);
\draw[color=gray,dotted] (4,1)--(4,7);
\draw[color=red,very thick] (2,7) -- (4,7);
\draw[color=red,very thick] (4,3) -- (4,7);

\draw (4.5, 8.5) node {$q \in [1, 4k+2]$};
\draw(4.5, 0) node{(a)};
\draw[color = black!50!green] (4, 3) node {\Large $*$};

\draw[->] (11,1) -- (11,8);
\draw (10.7,8.3) node {$p$};
\draw[->] (11,1) -- (18,1);
\draw (18.3,0.7) node {$q$};

\draw (11,7) -- (11.1,7);
\draw (11,7) node[anchor=east] {$\infty$};
\draw (11,3) -- (11.1,3);
\draw (11,3) node[anchor=east] {$2$};
\draw (11,2) -- (11.1,2);
\draw (11,2) node[anchor=east] {$1$};
\draw (11,1) node[anchor=east] {$0$};

\draw (12,1) -- (12,1.1);
\draw (12,1) node[anchor=north] {$1$};
\draw (13,1) -- (13,1.1);
\draw (13,1) node[anchor=north] {$2$};
\draw (15,1) -- (15,1.1);
\draw (15,1) node[anchor=north] {$4k+2$};
\draw (17,1) -- (17,1.1);
\draw (17,1) node[anchor=north] {$\infty$};
\draw (11,1) node[anchor=north] {$0$};

\draw[color=gray,dotted]   (11,1) -- (17,7);
\draw[color=gray,dotted]   (13,1) -- (13,3);
\draw[color=gray,dotted]   (17,1) -- (17,7);

\draw[color=gray,dotted] (11,7) -- (15,7);
\draw[color=gray,dotted] (11,3) -- (15,3);
\draw[color=gray,dotted] (15,1) -- (15,7);
\draw[color=gray,dotted] (12,1) -- (12,7);
\filldraw[fill=red!30,draw=red,very thick] (12,3) rectangle (15,7);

\draw[fill=blue!30,color=blue!30,very thin,fill opacity=0.5] (13,2) -- (17,2) --  (17,7) -- (13,3) -- (13,3);
\draw[color=blue,dashed,very thick] (13,3) -- (17,7);
\draw[color=blue,very thick]              (13,2) -- (13,3);
\draw[color=blue,dashed,very thick] (17,2) -- (17,7);
\draw[color=blue,dashed,very thick] (13,2) -- (17,2);

\draw (14.5, 8.5) node {$s \in [k+1, 2k] \cup \{k\}$};
\draw(14.5, 0) node{(b)};

\end{tikzpicture}
\caption{Adaptive minimax optimality over balls in
    $W^{s,p}_{0}(\T)$ or $B^{s,p'}_{p,0}(\T)$. (a) For
    $q \in [1, 4k+2]$, the MIND estimator with homogeneous Sobolev
    norm of order $k$ attains minimax optimal rates in terms of
    $L^q$-risk up to a log-factor, simultaneously in all spaces
    $W^{s,p}_{0}(\T)$ with smoothness parameters $s$, $p$ within
    the red region (``partial adaptation'').  By contrast,
    the Nemirovski's
    estimator $\hat{f}_{2,\eta}$ in~\eqref{eq:def:nemirovski:est} is minimax
    optimal up to a log-factor only for $W^{k,2}_{0}(\T)$,
    marked by a green asterisk. (b) For $s \in [k+1, 2k] \cup \{k\}$,
    MIND is optimal up to a log-factor in terms of $L^q$-risk over
    $S^{s,p}_{L}$ or $B^{s,p,p'}_{L}$ with
     parameters $p$, $q$ within the red region. Note that no
    linear estimator is optimal for  parameters $p$, $q$ in
    the blue region. 
\label{fig:optrates}}
\end{figure}
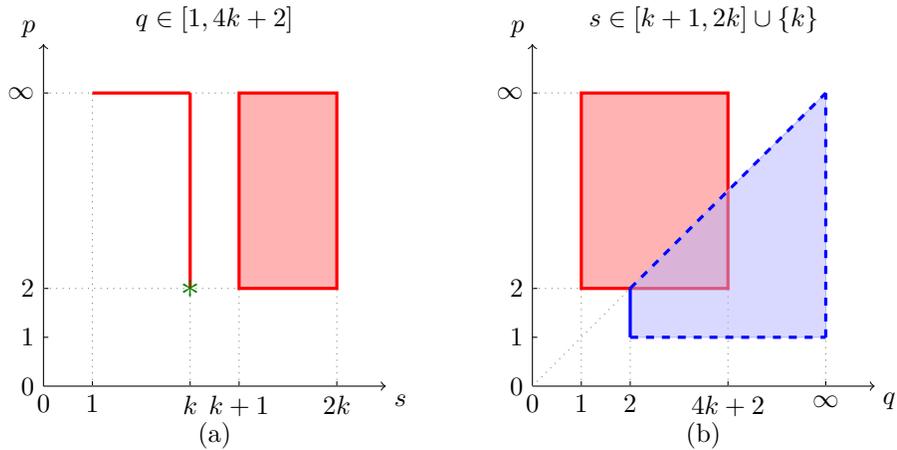

\section{Numerical Examples}\label{sec:num:example}

The MIND estimator defined by~\eqref{eq:constrainedprob} is the
solution to a {high dimensional} non-smooth convex optimization problem, due to the multiresolution norm. 
As mentioned in the introduction, there are {several}
efficient algorithms {nowadays} that are able to tackle such a problem. 
In this paper, we have chosen the
alternating direction method of multipliers (ADMM). It decomposes the
original problem~\eqref{eq:constrainedprob} into two sub-problems, the
first being the smoothing penalization problem, and the other the projection onto
the multiresolution ball. The latter one can be solved by the
Dykstra algorithm as introduced in~\citep{BoDy86}. We refer
to~\citep{FMM12b} for further details. It can be shown that the ADMM for
this problem converges linearly~\citep[cf.][]{DeYi13}, although the
theoretical rate might be very slow for large $k$, the smoothing order of
the regularization. We note that the
problem \eqref{eq:constrainedprob} can also be formulated as a
quadratic programming problem and solved, for instance, by the interior point
method~\citep[cf.][for example]{NNY94}. 

There are two practical concerns: 
The first is the choice of the system of cubes $\mathcal{B}$. The
general convergence rate results (cf.~Theorem~\ref{th:constrainedrate}) only require
that the system should be normal, see Remark~\ref{re:normSysExample}
for some examples. In the case of $d=1$, the concrete rates impose an
additional (but very weak) condition, namely that it should contain an
$m$-partition system for some $m$,  {i.e., $m$-regularity of $\mathcal{B}$}, 
see Theorem~\ref{th:rate_higherorder}. {For the examples in} our numerical simulations, 
we found that the system of all cubes, the system of all cubes with dyadic edge lengths,
and the $2$-partition system {perform comparably}. Therefore we {display the results for} the $2$-partition system in the following numerical experiments for
the sake of computational efficiency.

The other concern is the choice of the threshold $\gamma_n$. The
asymptotic theory only requires that $\gamma_n$ satisfies the condition 
{\eqref{eq:choice:gammaN}, which} is independent of the
interval system $\mathcal{B}$, and the
smoothness of the truth. In the finite sample situation, we recommend
a {refined} choice, which has a direct statistical interpretation, {cf. Section~\ref{subsec:var:stat:est}}. It
selects $\gamma_n$ as the $\alpha$-quantile of the multiscale
statistic, i.e.,
\begin{equation}\label{eq:def:gamma:significance}
\gamma_n(\alpha) := \inf\left\{\gamma:\Prob{\norm{\xi_n}_{\mathcal{B}}
    > \gamma} \le \alpha \right\}.
\end{equation}
This ensures that $f$ lies in the {confidence}
{set defined by the}
multiscale constraint in~\eqref{eq:constrainedprob} with probability
at least $1-\alpha$.
Thus we have 
\begin{equation}\label{eq:smoothControl}
\Prob{\norm{D^k\hat{f}_{\gamma_n}}_{L^2} \le \norm{D^k f}_{L^2}} \ge 1 - \alpha,
\end{equation}
that is, the MIND estimator is smoother than the truth with probability {at least}
 $1-\alpha$. {The asymptotic distribution of $\norm{\xi_n}_{\mathcal{B}}$ is, under general assumptions, a Gumbel law (after proper rescaling), see~\citep{Ka11,HalMun13}.} If $\mathcal{B}$ consists of all the cubes {and $\xi_n$ is standard Gaussian,}  then
\[ 
\gamma_n(\alpha) \sim \sqrt{2d\log n} +\frac{\log(d\log n) + \log J_d - 2\log\log({1}/{\alpha})}{2\sqrt{2d\log n}} \quad\text{ as } n \to \infty,
\]
where $J_d \in (0,\infty)$ is a constant. Although this violates the
condition {\eqref{eq:choice:gammaN} when $d =1$,} 
the asymptotic analysis in this paper
still {holds} for $\gamma_n(\alpha_n)$ if $\alpha_n \to 0$ sufficiently
fast, which might even possibly improve the rates, {in terms of the log-factor and the constant.} The estimation of
$\gamma_n$ can be done by Monte-Carlo simulations and is needed only
once for a fixed size of measurements and a fixed system of cubes.
The number of Monte-Carlo simulations is chosen as $10^{5}$ in the following examples. 

In the following simulations, we only consider the one-dimensional case $d=1$, and assume the noise is i.i.d. Gaussian with a known variance $\sigma^2$. In practice, one can easily {pre-}estimate $\sigma^2$, see e.g.~\citep{Rice84,HKT90,DMW98} among other references.

\begin{figure}[!h]
\begin{adjustbox}{addcode={\begin{minipage}{0.95\width}}{\caption{%
\label{fig:ConNem:DJsig}
Comparison of Nem in~\eqref{eq:def:nemirovski:est}, SS in~\eqref{eq:smooth:spline}, and the  MIND in~\eqref{eq:constrainedprob} (number of samples $n = 2^{11}$, noise level $\sigma = 0.12\norm{f}_{L^2}$).
}\end{minipage}},rotate=90,center}
\includegraphics[width= 1\textheight]{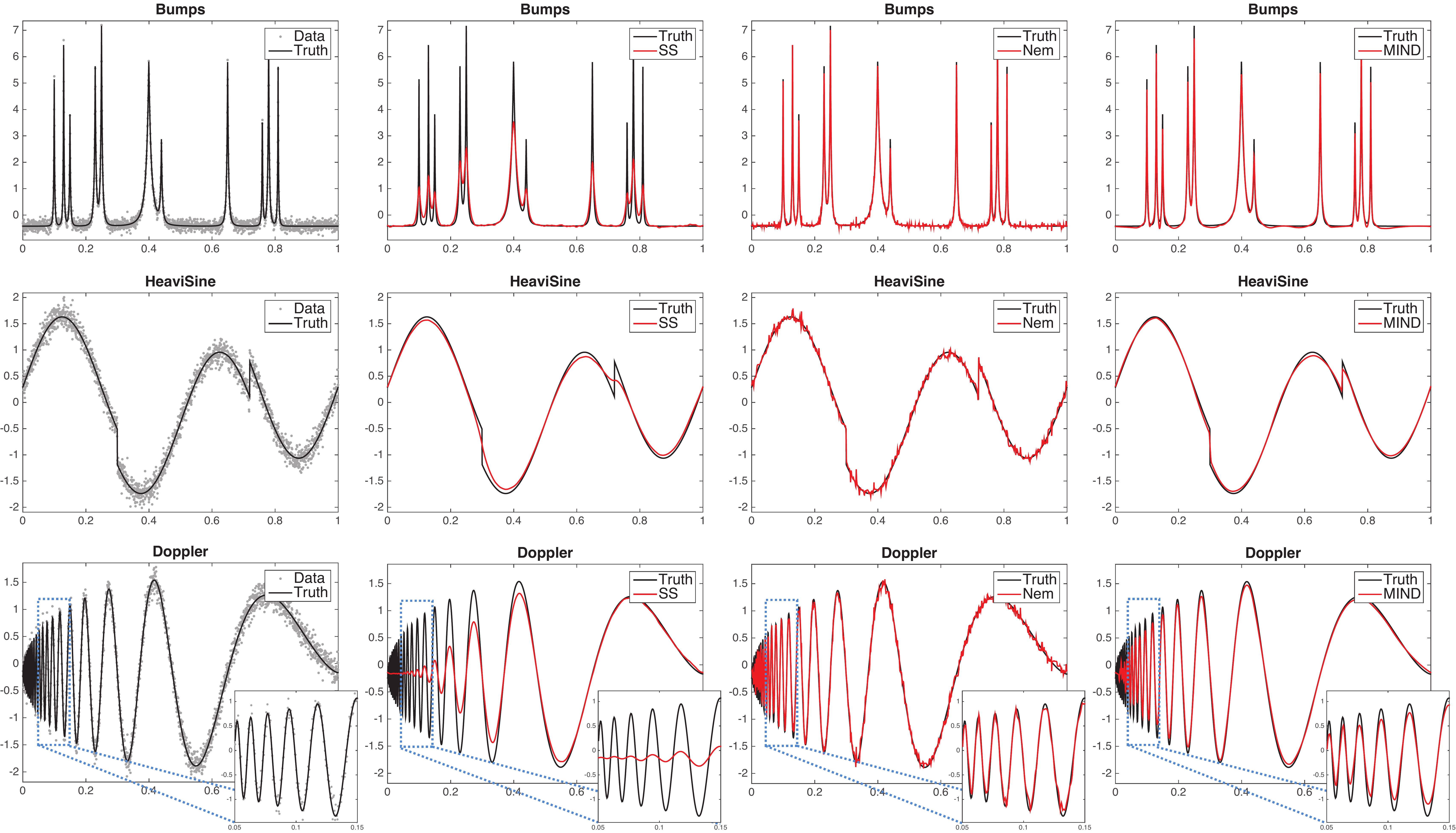}  %
\end{adjustbox}
\end{figure}

\subsection{Comparison study}

We now investigate the performance of MIND $\hat{f}_{\gamma_n(\alpha)}$ on spatially variable functions, Bumps, HeaviSine, and Doppler~\citep{DoJo94}, and compare it with the smoothing spline estimator (SS) $\hat{f}_{\lambda}$,  defined as the solution of 
\begin{equation}\label{eq:smooth:spline}
\min_{{f}} \norm{S_n {f} - y_n}_{\ell^2} + \lambda \norm{D^k {f}}_{L^2},
\end{equation}
and the Nemirovski's estimator (Nem) $\hat{f}_{2,\eta}$ in~\eqref{eq:def:nemirovski:est} with $p=2$. 
We choose $k = 1$ for all three estimators. The parameter $\alpha$ in MIND is set to 0.1, $\lambda$ in SS is tuned manually to give the best visual quality, and $\eta$ in Nem is chosen as the oracle $\norm{Df}_{L^2} (= : \eta_0)$, which is numerically
estimated using discrete Fourier transform. 

\begin{figure}[!h]
\centering
\includegraphics[width=0.9\textwidth]{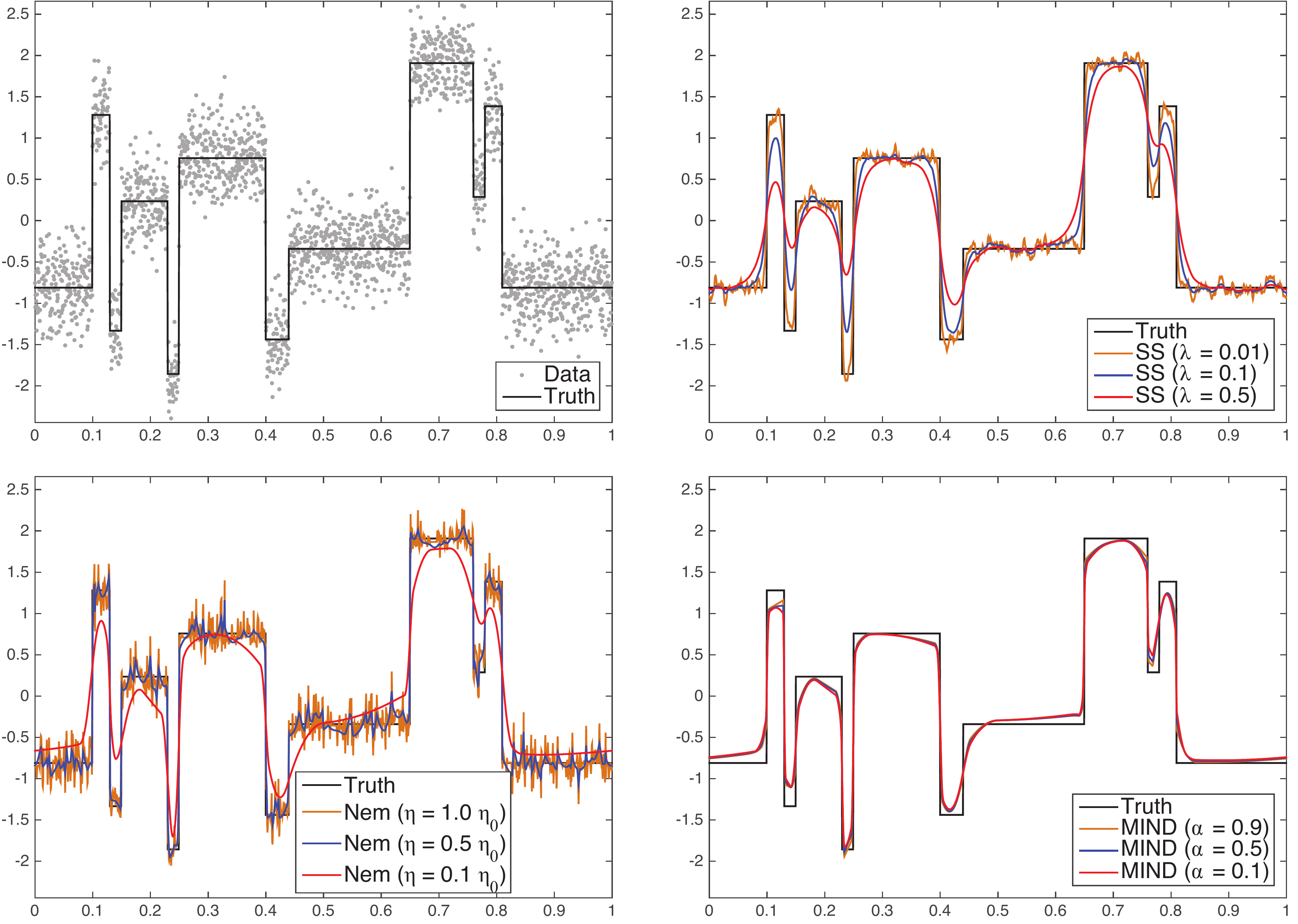}
\caption{{Impact of parameter choices: various $\eta$ for Nem  $\hat{f}_{2,\eta}$ in~\eqref{eq:def:nemirovski:est}, various $\lambda$ for SS $\hat{f}_{\lambda}$ in~\eqref{eq:smooth:spline}, and various $\alpha$ for the  MIND $\hat{f}_{\gamma_n}$ in~\eqref{eq:constrainedprob} with $\gamma_n = \gamma_n(\alpha)$ in~\eqref{eq:def:gamma:significance}. (number of samples $n = 2^{11}$, noise level $\sigma = 0.3\norm{f}_{L^2}$, and $\eta_0 = \norm{D^k f}_{L^2}$).}\label{fig:robustness} }
\end{figure}

The simulation results are summarized in Figure~\ref{fig:ConNem:DJsig}. One can see that MIND detects a large number of features {at} various scales of smoothness, and performs best on all the test signals. By contrast, SS with the ``optimal'' parameter recovers only a narrow range of scales of smoothness; for instance, on the Doppler signal, it works well for the smoother part (on $[0.5, 1]$), but deteriorates fast as the signal gets  more oscillatory. 
The Nem  with oracle $\eta (= \eta_0)$ is still very noisy on each test signal. We note that convex duality~(cf. Section~\ref{subsec:var:stat:est}) implies that there is a one-to-one correspondence between MIND and Nem as long as the different parameters are not unreasonably large.  The Nem will reproduce the results by MIND if we choose {as the threshold $\eta$,} $0.8 \eta_0$ for Bumps, $0.3\eta_0$ for HeaviSine, and $0.6 \eta_0$ for Doppler. {This means that,} even if $\eta_0 \equiv \norm{Df}_{L^2}$ is known exactly, one cannot find a universal threshold $\eta$ for Nem, which {explains our numerical findings.}

\subsection{Robustness {and stability}}
We examine the robustness of Nem, SS, and MIND, with respect to the choice of parameters, and the smoothness assumption, on the Blocks signal~\citep{DoJo94}, which is not even continuous, and hence falls not into the domain of our estimator. From Figure~\ref{fig:robustness} we find that the MIND estimator is rather robust to the choice of significance level $\alpha$, while Nem and SS are much more sensitive. Besides, MIND recovers the truth quite well with the correct number of local extrema, and slight distortion near change-points. As we already noted before, the performance SS is restricted to some fixed scale of smoothness. In contrast, Nem with a proper choice of threshold $\eta$ adapts to a wider range of smoothness scales, which is due to its relation to MIND via duality. Thus, this study again confirms that MIND  is practically preferable over Nem and SS. 

\begin{figure}[!h]
\centering
\includegraphics[width=0.65\textwidth]{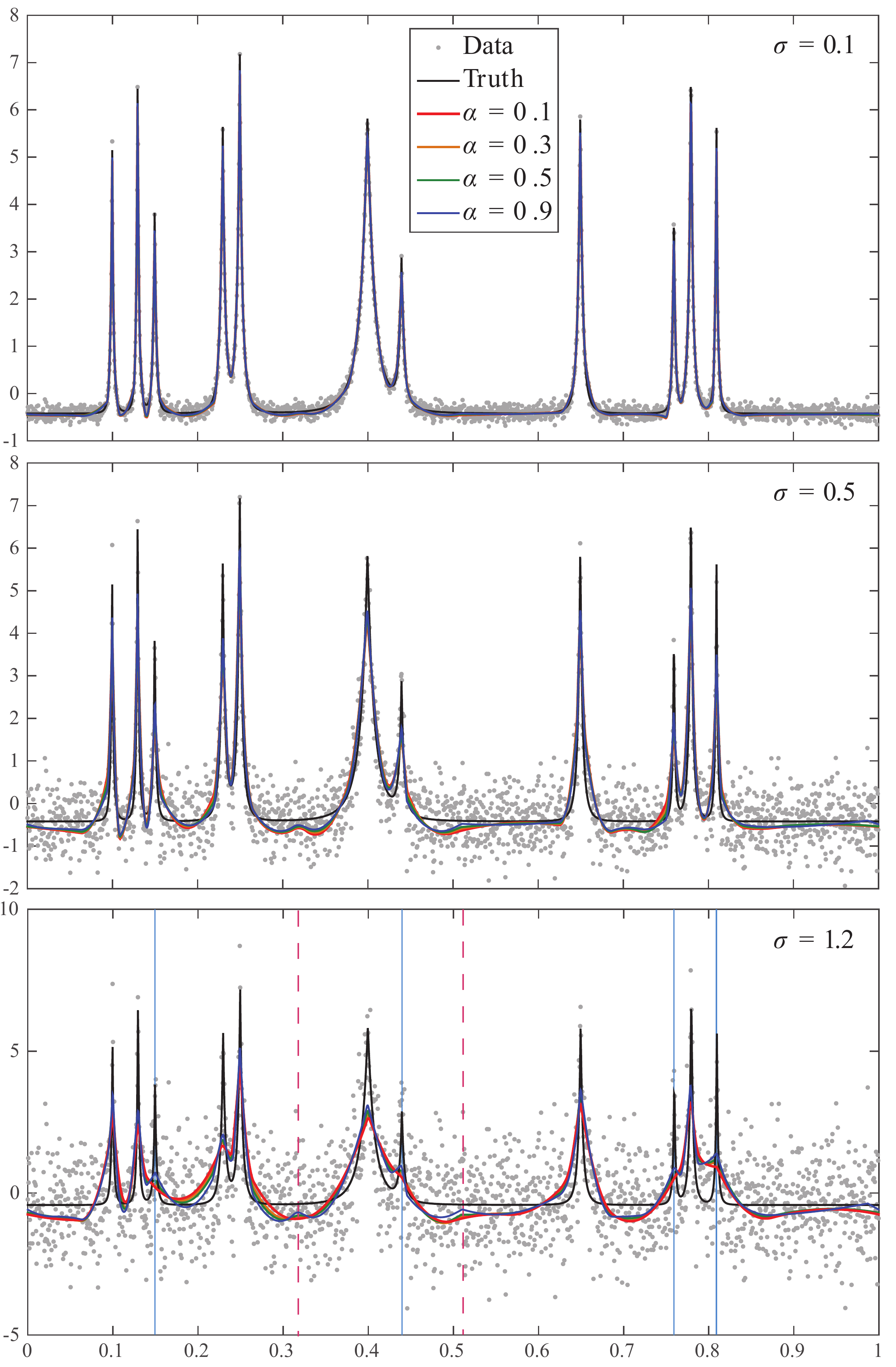}
\caption{{Stability of MIND in significance level $\alpha$ and noise level $\sigma$. The reconstructions by MIND $\hat{f}_{\gamma_n}$ with $\gamma_n = \gamma_n(\alpha)$ for a range of $\alpha$'s are shown, together with the true signal and noisy data, in the cases of different noise levels (number of samples $n = 2^{11}$). }\label{fig:stability} }
\end{figure}

{Now, we continue to consider the impact of significance level on the performance of the MIND estimator.  Exemplarily, we choose Bumps as the test signal for different noise levels. In Figure~\ref{fig:stability}, it shows that MIND with various choices of significance levels perform almost identically well in the case of low noise level ($\sigma = 0.1$) and medium noise level ($\sigma = 0.5$). However, in the high noise level ($\sigma = 1.2$) case, MIND with larger $\alpha$ tends to detect more bumps. For example, MIND ($\alpha =0.9$) recovers 6 more bumps than MIND ($\alpha = 0.1$), four out of which are actually correct (marked by vertical blue lines), while 2 false bumps are detected (marked by vertical red dashed lines, in the bottom panel of Figure~\ref{fig:stability}). Recall that the significance level $\alpha$ can be interpreted as an error control in the sense of~\eqref{eq:smoothControl}. Thus, the additional power by an increased significance level comes at the expense of a lower confidence about the inference. 
}

Note additionally that all the test signals considered so far are not strictly periodic, so the simulations also reveal that MIND is not {too} sensitive to the periodicity assumption. In practice, one can extend a non-periodic function to a periodic one by symmetric extension, see for instance~\citep{Mal09}. 
  
\begin{figure}[!ht]
\centering
\includegraphics[width=0.9\textwidth]{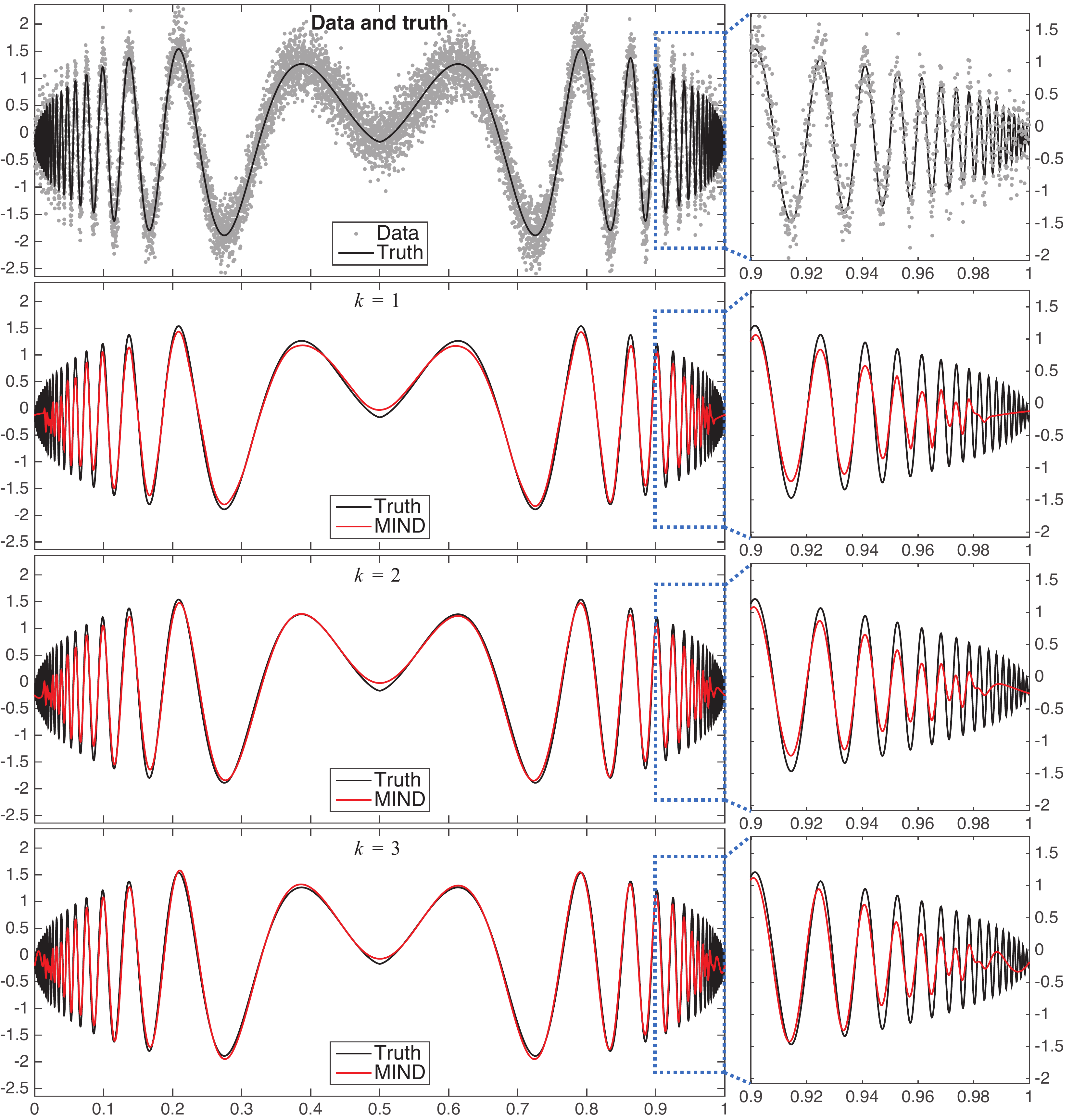}
\caption{{Various choices of $k$ for MIND in~\eqref{eq:constrainedprob} (number of samples $n = 2^{13}$, noise level  $\sigma = 0.3\norm{f}_{L^2}$).}\label{fig:Con:impact:k}}
\end{figure}

\subsection{Choice of smoothness order}\label{subsec:choice:k}
Now we explore the choice of smoothness parameter $k$ in the regularization term for the MIND estimator. The Doppler with symmetric extension (see Figure~\ref{fig:Con:impact:k}) is chosen as the test signal.
The significance level for MIND is set to $\alpha = 0.1$. Figure~\ref{fig:Con:impact:k} shows that MIND detects more features of different smoothness scales as $k$ increases, namely $24$ peaks for $k = 1$, $26$ for $k = 2$, and $27$ for $k = 3$. Meanwhile, the height of the peaks gets more accurate for larger $k$.  This is in accordance with our theoretical finding that the adaptation range increases with $k$, see Section~\ref{subsubsec:partial:adapt}. As already mentioned, one should, however, notice that the optimization problem becomes numerically {more} ill-conditioned as $k$ increases.

\section{Discussion}\label{sec:discuss}

In this paper, 
we have {introduced} a constrained variational estimator, MIND, which minimizes the $L^2$-norm of the
$k$-th order derivatives, $k$ being the anticipated smoothness of the 
function to be recovered,
subject to the constraint that the multiresolution norm of the residual
is bounded by some parameter $\gamma_n$ depending on the sample size $n$.
The idea behind this approach is that this norm effectively allows
to differentiate between smooth functions and noise, {as} the multiresolution norm of a continuous function is of the
order of $\sqrt{n}$, while the expected multiresolution norm of a sample of
independent sub-Gaussian noise is of the order of $\sqrt{\log n}$.
If we therefore use a threshold parameter $\gamma_n\sim (\log n)^r$ with
$r > 1/2$ we can expect that, for a sufficiently large sample size,
our  estimator, MIND, will be close to the true function, while the residuals
consist mostly of noise.

The main theoretical contribution of this paper was to underpin the already known empirically good performance {of MIND in several special cases} by some theoretical evidence. 
{For general dimension $d$,} from an interpolation inequality for the multiresolution norm and Sobolev norms (Proposition~\ref{prop:interpolation_regression}), we derive asymptotic convergence rates provided
 that $f \in H^k_0(\T^d)$ and one applies regularization with the homogeneous $H^k$-norm (see Remark~\ref{rem:natural:bnd:dist:fun}). Moreover, these rates turn out to be minimax optimal up to a logarithmic factor.
In order to derive convergence rates for different smoothness classes,
we have adapted the concept of approximate source conditions, to our statistical setting. These
are known to be a useful tool for the derivation of rates for deterministic
inverse problems. 
However, these conditions are quite abstract, and it is not immediately
clear how they relate to more tangible properties of $f$.

In the one-dimensional setting, a much more detailed analysis {is possible}.
Here the abstract conditions for convergence rates {can be related}
to approximation properties of splines. Mainly we have shown that the rates
depend on how well the $k$-th derivative of the function $f$ can be approximated
by $B$-splines with coefficients that are small with respect to the dual
multiresolution norm. Using results from approximation theory, we were
able to translate the approximate source conditions into very general
smoothness conditions for the function $f$. Mainly this gives us
optimal convergence rates for a function $f \in
H^{s}_{0}(\T)$ with $k+1 \le s \le 2k$. More general, we have
obtained with this argumentation convergence rates for functions $f$
contained in the fractional order Sobolev space
$W^{s,p}_{0}(\T)$ with $k+1\le s \le 2k$, and the rates are
again optimal as long as $p \ge 2$. Moreover, the same results hold
for comparable Besov spaces. While these results are only concerned
with functions $f$ that are of higher regularity than assumed
a-priori, it is also possible to derive rates for the case where $f$
is of lower regularity, that is, where the prior assumption that $f\in
H^k_0(\T)$ fails. The idea here is to approximate $f$ by a
spline of higher regularity and then to apply the higher order
convergence rate results to this spline. The final rate then results
from a trade-off between the approximation power of the spline and the
higher order convergence rate. With this technique, one obtains
optimal convergence rate for the lower order setting $f \in
W^{s,\infty}_0(\T)$ with $1 \le s \le k$.

It is important to note here that the choice of the parameter
$\gamma_n$ is independent of the actual smoothness of $f$. 
{This is why} MIND
yields (up to a logarithmic factor) simultaneously optimal convergence
rates for a range of smoothness classes {(with smoothness 
order $s \in [1,k] \cup [k+1, 2k]$)}, making it truly an adaptive
method. Additionally, the numerical results indicate that MIND
appears to be fairly robust with respect to the actual choice of the
parameter $\gamma_n$ for a given sample size $n$, further enhancing
its practical applicability.

There are several questions still open concerning MIND for nonparametric
regression. First of all, almost all concrete results concerning
convergence rates in this paper were derived for a one-dimensional
setting. In higher dimensions we only have the (somehow generic)
result mentioned in Remark~\ref{rem:natural:bnd:dist:fun} that gives us an optimal
convergence rate if our guess for the smoothness class of $f$ is
correct. It is, however, not at all obvious how to obtain rates for
higher order smoothness classes. In the one-dimensional case, the
method we used relied both on approximation results using $B$-splines
and on estimates for the dual multiresolution norm of the coefficients
of these $B$-splines. In higher dimensions, we expect that similar
results for polyharmonic splines would be required, but it is not
clear which basis splines have to be used {(cf.~\eqref{eq:dn:basis:version} 
and~\eqref{eq:dn:frame:version}). Moreover, in the literature, there 
is few results on the size of approximation coefficients, which are 
necessary for our analysis. } Similarly, the method we
have used for the derivation of the lower order convergence rates
relies intrinsically on spline approximation, which, again, makes the
generalization to higher dimensions difficult.

Also in the one-dimensional case there are several interesting open
questions. Our results only apply to a periodic setting with functions
that have zero mean. The main reason for the restriction to periodic
functions is that this avoids having to deal with boundary conditions
that would have to be taken into account in non-periodic cases. 
{
The restriction to functions with zero mean on the other hand is to simplify 
norms of Sobolev spaces. Dropping the requirement of zero mean
is possible; for instance, one can instead use $\hat f_{\gamma_n}^0 + \bar y_n$, 
with $\hat f_{\gamma_n}^0$ the MIND~\eqref{eq:constrainedprob} applied to $(y_n - \bar y_n)$.   
The same convergence rates still hold true by the analysis  
presented in this paper, and by the fact that $\bar y_n$ converges to the mean of the truth, $\int_{[0,1)} f(x) dx$, at a
 parametric rate, $\mathcal{O}_{\mathbb{P}}(1/\sqrt{n})$.
}
Most importantly, we are concerned with the gap between
$H^{k}_0(\T)$ and $H^{k+1}_0(\T)$. It seems reasonable
to assume that MIND is asymptotically optimal
also for functions in $H^s_0(\T)$ with $k < s < k+1$, but the
methods we have used for the derivation of the different rates appear
not to be applicable to this case. Also, while we do obtain
convergence rates for functions $f\in W^{s,p}_0(\T)$ with $p <
2$, these rates are not optimal. Here we suspect that this is due to
the fact that we use the $L^2$-norm of the $k$-th order derivative for
regularization and that better rates could be obtained by using the
$L^1$-norm instead.

In our future work, we will try to extend our results to the cases
mentioned above, in particular to higher dimensional and non-periodic
settings. Additionally, we plan to consider a generalization to the
solution of ill-posed operator equations. In particular for
deconvolution problems we expect that very similar
results can be obtained as in this paper, as long as the convolution
kernel has sufficiently slowly decreasing Fourier coefficients.

\appendix

\section{Proof of Theorem~\ref{th:constrainedrate}}\label{app:constrainedrate}

The assumption $\norm{S_nf -y_n}_{\mathcal{B}}\le \gamma_n$ implies that
$f $ is admissible for the minimization
problem~\eqref{eq:constrainedprob}, which in turn implies that
\[
\frac{1}{2}\norm{D^k \hat{f}_{\gamma_n}}_{L^2}^2 \le \frac{1}{2}\norm{D^k f }_{L^2}^2.
\]
As a consequence, we obtain the estimate
\[
\begin{aligned}
 \frac{1}{2} \norm{D^k \hat{f}_{\gamma_n}-D^k f }_{L^2}^2 
  &= \frac{1}{2}\norm{D^k \hat{f}_{\gamma_n}}_{L^2}^2 - \frac{1}{2}\norm{D^k f }_{L^2}^2 -
  \inner{f }{\hat{f}_{\gamma_n}-f }_{H^k_0}\\
  &\le -\inner{f }{\hat{f}_{\gamma_n}-f }_{H^k_0}\\
  &= \min_t
  \min_{\norm{\omega}_{\mathcal{B}^*} \le t} \Bigl(
  \inner{S_n^*\omega - f}{\hat{f}_{\gamma_n} -
    f}_{H^k_0} - \inner{S_n^*
    \omega}{\hat{f}_{\gamma_n}-f}_{H^k_0}\Bigr)\\
  &\le \min_t
  \min_{\norm{\omega}_{\mathcal{B}^*} \le t} \Bigl(
  \norm{D^k S_n^*\omega - D^k f}_{L^2}\norm{D^k \hat{f}_{\gamma_n} -
    D^k f}_{L^2} \\
  &\qquad\qquad\qquad\qquad{}+
  \norm{\omega}_{\mathcal{B}^*}\norm{S_n \hat{f}_{\gamma_n} - S_n f}_{\mathcal{B}}\Bigr)\\
  &\le \min_t \Bigl( d_n(t) \norm{D^k\hat{f}_{\gamma_n}-D^kf }_{L^2} +
  t\norm{S_n\hat{f}_{\gamma_n}-S_nf }_{\mathcal{B}}\Bigr).
\end{aligned}
\]
Thus we have for every $t \ge 0$ the inequality
\[
\frac{1}{2}\norm{D^k \hat{f}_{\gamma_n}-D^k f }_{L^2}^2 \le 
d_n(t) \norm{D^k\hat{f}_{\gamma_n}-D^kf }_{L^2} +
t\norm{S_n\hat{f}_{\gamma_n}-S_nf }_{\mathcal{B}}.
\]
Since
\[
\norm{S_n\hat{f}_{\gamma_n}-S_nf }_{\mathcal{B}} \le 2\gamma_n,
\]
we obtain the inequality
\begin{equation}\label{eq:constrained_h1}
  \begin{aligned}
    \norm{D^k \hat{f}_{\gamma_n}-D^kf }_{L^2}
    &\le  d_n(t) + \sqrt{d_n(t)^2
      +2t\norm{S_n\hat{f}_{\gamma_n}-S_nf }_{\mathcal{B}}} \\
    &\le 2d_n(t) + (2t)^{1/2}\norm{S_n\hat{f}_{\gamma_n}-S_nf }_{\mathcal{B}}^{1/2}\\
    &\le 2d_n(t) + 2(\gamma_n t)^{1/2}
  \end{aligned}
\end{equation}
for every $t \ge 0$.
We now recall the interpolation inequality (see Proposition~\ref{prop:interpolation_regression})
\begin{multline}\label{ineq:reg_var}
  \norm{\hat{f}_{\gamma_n} - f}_{L^q} \le C \max\biggl\{ \frac{\norm{S_n(
      \hat{f}_{\gamma_n} - f)}_{\mathcal{B}}^{2\vartheta}}{n^\vartheta}\norm{D^k( \hat{f}_{\gamma_n} - f)}_{L^2}^{1-2\vartheta},\, \\
  \frac{\norm{S_n( \hat{f}_{\gamma_n} - f)}_{\mathcal{B}}}{n^{1/2}},\,\frac{\norm{D^k( \hat{f}_{\gamma_n} - f)}_{L^2}}{n^{\vartheta'}}\biggr\}
\end{multline}
for $n$ sufficiently large,
with some $C > 0$ and $\vartheta' = 2k\vartheta/d$.

If the maximum in~\eqref{ineq:reg_var} is attained at the first term,
the estimate~\eqref{eq:constrained_h1} implies that
\[
\begin{aligned}
  \norm{\hat{f}_{\gamma_n}-f }_{L^q}
  &\le
  C\frac{\norm{S_n\hat{f}_{\gamma_n}-S_n f}_{\mathcal{B}}^{2\vartheta}}{n^{\vartheta}}
  \norm{D^k \hat{f}_{\gamma_n}-D^k f}^{1-2\vartheta}_{L^2}\\
  &\le C \frac{(2\gamma_n)^{2\vartheta}}{n^\vartheta} \min_t \bigl(
  2d_n(t) + 2 (\gamma_n t)^{1/2}\bigr)^{1-2\vartheta}\\
  &\le 2C
  \frac{\gamma_n^{2\vartheta}}{n^\vartheta}\min_{t}\bigl( d_n(t) +
  (\gamma_n t)^{1/2}\bigr)^{1-2\vartheta}.
\end{aligned}
\]
On the other hand, if the maximum in~\eqref{ineq:reg_var} is attained
at the second term, we have
\[
\norm{\hat{f}_{\gamma_n} - f}_{L^q} \le C\frac{\norm{S_n( \hat{f}_{\gamma_n} - f)}_{\mathcal{B}}}{n^{1/2}}
\le 2C \frac{\gamma_n}{n^{1/2}}.
\]
Finally, if the maximum in~\eqref{ineq:reg_var} is attained
at the third term, we have
\begin{multline*}
\norm{\hat{f}_{\gamma_n} - f}_{L^q} \le C\frac{\norm{D^k( \hat{f}_{\gamma_n} - f)}_{L^2}}{n^{\vartheta'}}
\le C \min_t\frac{2d_n(t) + 2(\gamma_nt)^{1/2}}{n^{\vartheta'}} \\
\hfill\le 2C\frac{1}{n^{\vartheta'}}\min_{t}\bigl( d_n(t) + (\gamma_n t)^{1/2}\bigr).\hfill \ensuremath{\square}
\end{multline*}

\section{Proof of Theorem~\ref{th:rate_expect}}\label{app:expectedrate}

Denote in the following
\[
p(t) := \Prob{\norm{\xi_n}_{\mathcal{B}} \le t}.
\]
Using Theorem~\ref{th:constrainedrate}, we see that we can estimate,
for $n$ sufficiently large,
\begin{multline}\label{eq:expected_h1}
  \E{\norm{\hat{f}_{\gamma_n}-f}_{L^q}}
  \le \Prob{\norm{\xi_n}_{\mathcal{B}} \le \gamma_n} C n^{-\mu(1-2\vartheta)-\vartheta} {(\log n)}^{r(1+2\vartheta)/2}\\
  + \int_{\gamma_n}^\infty \sup\set{\norm{\hat{f}_{\gamma_n}-f}_{L^q}}{\norm{\xi_n}_{\mathcal{B}} = t} dp(t).
\end{multline}

In the following, we will show that the second term on the right hand side
of~\eqref{eq:expected_h1} tends to zero faster as $n\to\infty$. To that end, we observe first that
the Sobolev embedding theorem~\citep[][Theorem 4.12]{AF03} and the
Poincar\'e inequality~\citep[][Theorem 4.4.2]{Zie89} imply that
\[
\norm{\hat{f}_{\gamma_n}-f}_{L^q}
\le \norm{\hat{f}_{\gamma_n}}_{L^q} + \norm{f}_{L^q}
\le C\norm{D^k\hat{f}_{\gamma_n}}_{L^2} + \norm{f}_{L^\infty}
\]
for some constant $C$ depending only on $q$, $d$ and $k$.
Moreover, by construction, we have
\[
\norm{D^k\hat{f}_{\gamma_n}}_{L^2} \le \norm{D^k g}_{L^2}
\text{ for all } g \text{ satisfying }
\norm{S_n g - S_n f - \xi_n}_{\mathcal{B}} \le \gamma_n.
\]

Now let $h\in H^k(\R^d)$ be such that $h(0) = 1$,
$\int_{\R^d} h(x)\,dx = 0$, and $\supp h \subset [-1/2,1/2]^d$.
Define moreover, for $n\in\N$ and $x \in \Gamma_n$, the function
$h_{n,x}\colon\T^d\to\R$ by
\[
h_{n,x}(y) = h(n^{1/d}(y-x)).
\]

Let now $n$ and $\xi_n\in\R^{\Gamma_n}$ be fixed and define
\[
g := \sum_{x\in\Gamma_n} (f(x)+\xi_n(x)) h_{n,x}.
\]
Since the functions $h_{n,x}$, $x \in \Gamma_n$, have pairwise disjoint
supports, it follows that
\begin{align*}
\norm{D^k g}_{L^2} = & \sum_{x\in\Gamma_n} \abs{f(x)+\xi_n(x)}\norm{D^k h_{n,x}}_{L^2}
= \sum_{x\in\Gamma_n} \abs{f(x)+\xi_n(x)}  n^{\frac{2k-d}{2d}}\norm{D^k h}_{L^2}\\
= & n^{\frac{2k-d}{2d}}\norm{S_n f + \xi_n}_{\ell^1} \norm{D^k h}_{L^2}
\le n^{\frac{2k+d}{2d}} \bigl(\norm{f}_{L^\infty} + \norm{\xi_n}_{L^\infty}\bigr)\norm{D^k h}_{L^2}.
\end{align*}
From the inequality $\norm{\xi_n}_{L^\infty} \le \norm{\xi_n}_{\mathcal{B}}$ we thus obtain
that, for some constant $C$ only depending on $q$, $d$, and $k$,
\[
\sup\set{\norm{\hat{f}_{\gamma_n}-f}_{L^q}}{\norm{\xi_n}_{\mathcal{B}} = t}
\le C n^{\frac{2k+d}{2d}} (\norm{f}_{L^\infty} + t).
\]

As a consequence, we can estimate the last term in~\eqref{eq:expected_h1} by
\begin{multline*}
  \int_{\gamma_n}^\infty \sup\set{\norm{\hat{f}_{\gamma_n}-f}_{L^\infty}}{\norm{\xi_n}_{\mathcal{B}} = t} dp(t)
  \le \int_{\gamma_n}^\infty Cn^{\frac{2k+d}{2d}}(\norm{f}_{L^\infty}+t)\,dp(t)\\
  = Cn^{\frac{2k+d}{2d}}(\norm{f}_{L^\infty}+\gamma_n)(1-p(\gamma_n)) - Cn^{\frac{2k+d}{2d}}\int_{\gamma_n}^\infty (p(t)-1)\,dt.
\end{multline*}
From Proposition~\ref{prop:mulitresolution_probability} we obtain that
\[
(1-p(t)) \le 2n^2 e^{-\frac{t^2}{2\sigma^2}}
\]
for sufficiently large $n$. Thus we see that
\begin{align*}
  &\int_{\gamma_n}^\infty \sup\set{\norm{\hat{f}_{\gamma_n}-f}_{L^\infty}}{\norm{\xi_n}_{\mathcal{B}} = t} dp(t)\\
  \le& 2Cn^{\frac{2k+5d}{2d}}(\norm{f}_{L^\infty}+\gamma_n)e^{-\frac{\gamma_n^2}{2\sigma^2}} 
  + 2Cn^{\frac{2k+5d}{2d}}\int_{\gamma_n}^\infty e^{-\frac{t^2}{2\sigma^2}}\,dt
  \le C'n^{\frac{2k+5d}{2d}} \gamma_n e^{-\frac{\gamma_n^2}{2\sigma^2}}
\end{align*}
for sufficiently large $n$.
Now the choice of $\gamma_n$ implies that
\[
n^{\frac{2k+5d}{2d}}\gamma_n e^{-\frac{\gamma_n^2}{2\sigma^2}} = \mathcal{O}(n^{-\eps})
\]
as $n \to \infty$ for some $\eps > 0$. This shows that the second
term in~\eqref{eq:expected_h1} tends to zero faster as $n\to\infty$,
which concludes the proof of Theorem~\ref{th:rate_expect}. \hfill \ensuremath{\square}

\section{Proof of Proposition~\ref{pr:distance_1d}}\label{app:proof_dist_1d}

The main idea of the proof of Proposition~\ref{pr:distance_1d}
is to approximate the function $f$ with splines, the coefficients
of which are (relatively) small with respect to the dual multiresolution
norm. As a preparation, we will need several results concerning
approximation properties of splines, most of which are well known
in approximation theory, and a result that allows us to bound
the dual multiresolution norm of a spline function.

\subsection{Approximation properties of splines}

\begin{proposition}[Condition number of B-splines]\label{prop:size_coef}
Assume $\{Q^m_i(x), i=0,\ldots,n-1\}$ is the family of normalized B-splines in $\mathcal{S}_m(\Gamma_n;\T)$. Then for any $c_i\in\R, i=0$, $\ldots,n-1$,
\begin{equation}\label{eq:cond:bspline}
\norm{(c)_{i=0}^{n-1}}_{p}\le m 2^mn^{1/p}\norm{\sum_{i=0}^{n-1}c_iQ^m_i}_{L^{p}} \quad \text{ for } 1 \le p \le \infty.
\end{equation}
\end{proposition}

\begin{proof}
Let us first consider  $1 \le p < \infty$. By $\{\tilde{Q}^m_i\}_{i=-m+1}^{ln-1}$ we denote the normalized B-splines on the real line with equally spaced knots $\{{(-m+1)}/{n},  (-m+2)/n, \ldots, (ln+m-1)/n\}$. Let $\tilde{c}_i := c_{i \text{ mod } n}$ for $i = -m+1,\ldots, ln-1$. It is known from~\citep[Theorem 1]{SchSh99} that 
\[
\norm{(\tilde{c}_i)_{i=-m+1}^{ln-1}}_{p} \le m2^mn^{1/p}\norm{\sum_{i=-m+1}^{ln-1}\tilde{c}_i\tilde{Q}^m_i}_{L^p} \quad \text{ for any } l \in \N. 
\]
It implies
\begin{multline*}
l\norm{(c_i)_{i=0}^{n-1}}^p_p + \norm{(c_i)_{i = n-m+1}^{n-1}}_p^p
\le n(m2^m)^p \Biggl(l\norm{\sum_{i=0}^{n-1}c_i Q^m_i}^p_{L^p} 
+  \norm{\sum_{i = n-m+1}^{n-1}Q_i^m \mathbf{1}_{[0, \frac{m-1}{n})
      \cup [\frac{n-m+1}{n},1)}}_{L^p}^p\Biggr)
\end{multline*}
or
\begin{multline*}
\norm{(c_i)_{i=0}^{n-1}}^p_p + \frac{1}{l}\norm{(c_i)_{i = n-m+1}^{n-1}}_p^p \le n(m2^m)^p \Biggl(\norm{\sum_{i=0}^{n-1}c_i Q^m_i}^p_{L^p} 
 +  \frac{1}{l}\norm{\sum_{i = n-m+1}^{n-1}Q_i^m \mathbf{1}_{[0, \frac{m-1}{n}) \cup [\frac{n-m+1}{n},1)}}_{L^p}^p\Biggr).
\end{multline*}
By letting $l \to \infty$, we obtain~\eqref{eq:cond:bspline} for $1\le p < \infty$.  Then, the case $p=\infty$ follows by taking $p \to\infty$.
\end{proof}

\begin{remark}
This result is a generalization of a known result for splines on $\R$~\citep[see][]{SchSh99} to periodic splines. 
\end{remark}

\begin{proposition}[Boundedness of $L^2$-projector]\label{prop:Linfty_norm_L2_proj}
Let $P_{\mathcal{S}}$ be the orthogonal projector onto $\mathcal{S}_m(\Gamma_n;\T)$ in the topology of $L^2(\T)$. Then there is a constant $C$ depending only on $m$ such that 
\[
\norm{P_{\mathcal{S}}u}_{L^{p}} \le C\norm{u}_{L^{p}} \quad \text{for any } u \in L^{p}(\T) \text{ and } 1 \le p \le \infty.
\]
\end{proposition}

\begin{proof}
By \citep[Theorem 2.11]{AF03} and \citep[][Theorem 6.30]{Schumaker2007}, 
it is sufficient to prove this assertion only for $p = 1$ and $p = \infty$. 

Consider first the case $p = \infty$.  Let $Q^m_i \in
\mathcal{S}_m(\Gamma_n;\T)$ be the normalized B-splines, and
$R_i^m:=nQ^m_i$ implying that $\norm{R^m_i}_{L^1} = 1$. If $P_{\mathcal{S}}f = \sum_{i=0}^{n-1} a_i Q_i^m$, then 
$$
\sum_{j=0}^{n-1} a_j \inner{Q_j^m}{R_i^m} = \inner{f}{R_i^m}, 
$$
that is, we have an equation of the form $G a = b$ with $a := (a_i)_{i=0}^{n-1}$, $b:=(\inner{f}{R_i^m})_{i=0}^{n-1}$ and $G:=(\inner{R_i^m}{Q_j^m})_{i,j}$. Note that 
$$
\norm{b}_{\infty} = \max_{i}\abs{\inner{f}{R_i^m}}\le \max_{i}\norm{f}_{L^{\infty}}\norm{R_i^m}_{L^1} = \norm{f}_{L^\infty}.
$$ 
This implies that
\[
\norm{P_{\mathcal{S}}f}_{L^{\infty}} = \norm[B]{\sum_{i=0}^{n-1}a_iQ_i^m}_{L^\infty} \le \norm{a}_{\infty} \le \norm{G^{-1}}_{\infty}\norm{b}_{\infty}\le \norm{G^{-1}}_{\infty}\norm{f}_{L^\infty}.
\]
It follows from~\citep{DeBoor12} that 
\[
\norm{G^{-1}}_{\infty} \le C_m
\]
for some constant $C_m$ depending only on $m$. Thus, $\norm{P_{\mathcal{S}}f}_{L^{\infty}} \le C_m \norm{f}_{L^\infty}$.

Next consider $p = 1$. Let $P_{\mathcal{S}}f =
\sum_{i=0}^{n-1}\tilde{a}_i R_i^m$, then $\sum_{j=0}^{n-1} \tilde{a}_j
\inner{R_j^m}{Q_i^m} = \inner{f}{Q_i^m}$, i.e., $G^t \tilde{a} =
\tilde{b}$, where $(\cdot)^t$ denotes transpose, $\tilde{a} :=
(\tilde{a}_i)_{i=0}^{n-1}$ and
$\tilde{b}:=(\inner{f}{Q_i^m})_{i=0}^{n-1}$. It follows from $\sum_{i}
Q_i^m = 1$ and $Q_i^m \ge 0$ that
$$
\norm{\tilde{b}}_1 = \sum_{i} \abs{\inner{f}{Q_i^m}} \le \sum_{i}\inner[b]{\abs{f}}{Q_i^m} = \inner[B]{\abs{f}}{\sum_{i}Q_i^m} = \norm{f}_{L^1}.
$$
Then
\begin{align*}
\norm{P_{\mathcal{S}}f}_{L^1} = & \norm[B]{\sum_{i=0}^{n-1} \tilde{a}_i R_i^m}_{L^1} \le \norm{\tilde{a}}_1 \le \norm{G^{-t}}_1\norm{\tilde{b}}_1 \\
= & \norm{G^{-1}}_{\infty}\norm{\tilde{b}}_{1}\le \norm{G^{-1}}_{\infty}\norm{f}_{L^1} \le C_m \norm{f}_{L^1}.
\end{align*}
That is, we obtain the assertion for $p=1$. 
\end{proof}

\begin{remark}
The above result (of periodic splines with equally spaced knots) is probably proven in 1970s, but we are not aware of the reference.  The proof we give here also shows the result for periodic splines with non-equally spaced knots, since \cite{DeBoor12} proved the boundedness of the inverse Gram matrix of B-splines for any knots. Similar results for non-periodic splines with arbitrary knots are originally proven in \citep{Shadrin2001}, and recently shortened in \citep{Golitschek2014}.
\end{remark}

\begin{proposition}[Approximation property]\label{prop:spline_approx}
Let $ 1 \le p, p',  q \le \infty$. 
There exists a linear operator $A:L^1_{0}(\T)\to\mathcal{S}_m(\Gamma_n;\T)$ such that for every $u\in L^1_{0}(\T)$
\[
\begin{aligned}
&\norm{u-Au}_{W^{r,q}_{0}} \le C_1\frac{\norm{u}_{B^{s,p'}_{p,0}}}{n^{s- r - (1/p-1/q)_+}} \quad \text{ with } 1\le s \le m,\, 0 \le r \le \lfloor s-1\rfloor,\\
&\norm{Au}_{W^{r,q}_{0}} \le C_2\frac{\norm{u}_{{B^{s,p'}_{p,0}}}}{n^{s-r-(1/p-1/q)_+}} \quad \text{ with } 1\le s \le \lceil s \rceil \le r \le m-1,
\end{aligned}
\]
where $C_1$, $C_2$ depend only on $m$, $p$. Moreover, both inequalities also hold for the Sobolev norm $\norm{\cdot}_{W^{s,p}_{0}}$ when $p = p'$ and $s \in \N$.
\end{proposition}

\begin{remark}
In the case of Sobolev norm $\norm{\cdot}_{W^{s,p}_{0}}, \, s\in\N$,  the assertions follow from {\citep[][Theorem~8.12]{Schumaker2007}.} Following the idea of the proof of {\citep[][Theorem~6.31]{Schumaker2007},} such results can be extended to Besov norms using {\citep[][Theorem 6.30]{Schumaker2007}.}
\end{remark}

\begin{proposition}[Finite differences and $W^{1,p}(\T)$]\label{prop:bnd:fd}
Let $h > 0$ and $1\le p \le \infty$. Then
\begin{equation}\label{eq:fd}
\norm{D_{h,+}f}_{L^p} = \norm{D_{h,-}f}_{L^p} \le h \norm{D f}_{L^p} \quad \text{ for } f \in W^{1,p}(\T). 
\end{equation}
\end{proposition}
\begin{proof}
The case of $p = \infty$ is obviously true. 
Now consider $1 \le p <\infty$. Since $\norm{D_{h,+}f}_{L^p} =
\norm{D_{h,-}f}_{L^p}$, it is sufficient to prove~\eqref{eq:fd} only
for $D_{h,+}$.  Note that for each $f\in W^{1,p}(\T)$ there is a
sequence of smooth functions $f_n$, such that $\norm{D_{h,+}f_n}_{L^p}
\to \norm{D_{h,+}f}_{L^p}$ and $\norm{D f_n}_{L^p} \to \norm{D
  f}_{L^p}$ as $n \to \infty$.  Therefore, we assume without loss of
generality that $f$ is a
smooth function. It follows from the equation $f(x+h)-f(x) = h\int_{0}^1
f'(x+th)dt$ that
\begin{align*}
\int_{0}^1\abs{f(x+h)-f(x)}^p dx & \le h^p\int_0^1\Bigl(\int_0^1\abs{f'(x+th)}dt\Bigr)^p dx \\
& \le h^p\int_0^1\int_0^1\abs{f'(x+th)}^p dt\, dx \\
& = h^p \int_0^1\int_0^1 \abs{f'(x+th)}^p dx\, dt \\
& = h^p \norm{f'}_{L^p}^p. 
\end{align*}
That is $\norm{D_{h,+}f}_{L^p} \le h\norm{D f}_{L^p}$.
\end{proof}

\subsection{Regular Systems}

Next we state two technical lemmas,
which allow us to estimate the dual multiresolution norm of piecewise
constant vectors in case the system $\mathcal{B}$ is $m$-regular.
These piecewise constant vectors will appear as spline coefficients
for certain approximation splines needed for the proof of
Proposition~\ref{pr:distance_1d}.

\begin{lemma}\label{le:cover_help}
  Assume that $m \in \N$, $m\ge 2$, and that $n \in \N$
  is written as
  \[
  n = \sum_{j=0}^r d_j m^j
  \qquad\text{ with } d_j \in \{0,\ldots,m-1\}
  \]
  and $r = \lfloor \log_m n \rfloor$.
  Then
  \[
  \sum_{j=0}^r d_j m^{j/2} \le (\sqrt{m}+1) \sqrt{n}.
  \]
\end{lemma}

\begin{proof}
  We prove this claim by induction over $r$.
  For $r = 0$ it is trivial.

  Now assume that the claim holds for $r$ and let $n$ be such that
  $\lfloor \log_m n \rfloor = r+1$. Then  
  \[
  \begin{aligned}
    &\Bigl(\sum_{j=0}^{r+1} d_j m^{j/2}\Bigr)^2\\
    = &\Bigl(\sum_{j=0}^r d_j m^{j/2}\Bigr)^2 + d_{r+1}^2 m^{r+1} + 2d_{r+1} m^{(r+1)/2}\sum_{j=0}^r d_jm^{j/2}\\
    \le& (\sqrt{m}+1)^2\sum_{j=0}^r d_jm^j + d_{r+1}^2 m^{r+1} + 2d_{r+1} m^{(r+1)/2}(\sqrt{m}+1) \Bigl(\sum_{j=0}^r d_j m^j\Bigr)^{1/2}\\
    \le &(\sqrt{m}+1)^2\sum_{j=0}^r d_jm^j + d_{r+1}^2 m^{r+1} + 2d_{r+1} m^{(r+1)/2}(\sqrt{m}+1) m^{(r+1)/2}\\
    = &(\sqrt{m}+1)^2\sum_{j=0}^r d_j m^j + \left(d_{r+1}+2(\sqrt{m}+1)\right)d_{r+1} m^{r+1}.
  \end{aligned}
  \]
  Since $d_{r+1} \le m-1$ and $m-1 + 2(\sqrt{m}+1) = (\sqrt{m}+1)^2$, this proves the assertion.
\end{proof}

\begin{lemma}\label{le:cover}
  Assume that the family $\mathcal{B}$ is $m$-regular
  for some fixed $m \ge 2$.
  Let now $I = \{i_0,i_0+1,\ldots,i_0+p-1\}/n \subset \Gamma_n$
  and define $c \in \R^{\Gamma_n}$ by $c_i = 1$ if $i \in I$
  and $c_i = 0$ if $i \not \in I$. Then
  \[
  \norm{c}_{\mathcal{B}^*} \le (\sqrt{m} + 1)\sqrt{2mp}.
  \]
\end{lemma}

\begin{proof}
  Let $r = \lceil\log_m n\rceil$.
  Let $\ell_-\in\N$ be maximal such that $\ell_- m^{-r} \le i_0/n$,
  and let $\ell_+\in\N$ be minimal such that $\ell_+ m^{-r} > (i_0+p-1)/n$.
  Then
  \[
  \ell_+-\ell_- < \frac{m^r}{n}(p-1) + 2 < mp.
  \]
  Now write
  \[
  \ell_- = \sum_{j=0}^r d_j^- m^j
  \qquad\text{ and }\qquad
  \ell_+ = \sum_{j=0}^r d_j^+ m^j.
  \]
  Let moreover $0 \le s \le r-1$ be maximal such that $d_s^- < d_s^+$
  and denote by $\hat{\ell}$ the minimal number of the form
  \[
  \hat{\ell} = \hat{d}_s m^s + \sum_{j=s+1}^r d_j^+ m^j
  \]
  such that $\ell_- \le \hat{\ell}$.

  Next we denote by $\mathcal{B}_1$ the collection of intervals
  of the form
  \begin{multline*}
  [\ell m^{k-r},(\ell+1)m^{k-r})
  \quad\text{ where } 0 \le k \le s-1,
  \text{ and } 
  \ell = \sum_{j = k+1}^r d_j^+ m^j + dm^k
  \text{ with } 0\le d < d_k^+.
  \end{multline*}
  Similarly, we denote by $\mathcal{B}_2$ the collection
  of intervals of the form
  \[
  [\ell m^{s-r},(\ell+1)m^{s-r})
  \qquad\text{ where } \ell = \hat{\ell} + dm^s
  \text{ with } 0 \le d < d_s^+ - \hat{d}_s.
  \]
  Then the intervals contained in $\mathcal{B}_1 \cup\mathcal{B}_2$
  form a disjoint cover of $[\hat{\ell}m^{-r},\ell_+m^{-r})$.

  Next we write
  \[
  \hat{\ell}-\ell_- = \sum_{j=0}^{s-1}\hat{d}_j^- m^j
  \]
  and denote by $\mathcal{B}_3$ the collection of intervals of the form
  \begin{multline*}
  \hat{\ell}m^{-r}-(\ell m^{k-r},(\ell+1)m^{k-r}]
  \quad\text{ where } 0 \le k \le s-1,
  \text{ and } 
   \ell = \sum_{j = k+1}^{s-1} \hat{d}_j^- m^j + dm^k
  \text{ with } 0\le d < \hat{d}_k^-.
  \end{multline*}
  Then the intervals contained in $\mathcal{B}_3$ form
  a disjoint cover of $[\ell_-m^{-r},\hat{\ell}m^{-r})$.
  
  Note in addition that by construction all of these intervals 
  are also contained in $\mathcal{B}$.
  Now denote $\hat{\mathcal{B}} = \mathcal{B}_1\cup\mathcal{B}_2\cup\mathcal{B}_3$
  and define $c_B := 1$ if $B \in \hat{\mathcal{B}}$ and $B\cap I \neq \emptyset$
  and $c_B := 0$ if $B \in\mathcal{B}\setminus\hat{\mathcal{B}}$
  or $B\in\hat{\mathcal{B}}$ and $B\cap I = \emptyset$.
  Then $c_i = \sum_{B \ni i} c_B$ for all $i$ and therefore
  \[
  \norm{c}_{\mathcal{B}^*} 
  \le \sum_{B\in\mathcal{B}} \abs{c_B} \sqrt{n(B)}
  \le \sum_{B\in\hat{\mathcal{B}}} \sqrt{n(B)}.
  \]
  Now note that
  \[
  \sqrt{n([\ell m^{k-r},(\ell+1)m^{k-r}),n)} \le m^{k/2}
  \qquad\text{ for all } 0 \le k \le r.
  \]
  Therefore Lemma~\ref{le:cover_help} and the definition of $\hat{\mathcal{B}}$
  imply that
  \[
  \begin{aligned}
    \norm{c}_{\mathcal{B}^*}
    &\le \sum_{k=0}^{s-1} d_k^+ m^{k/2} + (d_s^+-\hat{d}_s)m^{s/2} + 
    \sum_{k=0}^{s-1} \hat{d}_k^-m^{k/2} \\
    &\le (\sqrt{m} +1) \biggl(\Bigl((d_s^+-\hat{d}_s)m^s + \sum_{k=0}^{s-1} d_k^+ m^k\Bigr)^{1/2}
    + \Bigl(\sum_{k=0}^{s-1} \hat{d}_k^- m^k\Bigr)^{1/2}\biggr)\\
    &= (\sqrt{m} +1) \Bigl(\sqrt{\ell_+-\hat{\ell}}+\sqrt{\hat{\ell}-\ell_-}\Bigr)\\
    &\le (\sqrt{m}+1)\sqrt{2(\ell_+-\ell_-)}\\
    &\le (\sqrt{m}+1)\sqrt{2mp}.
  \end{aligned}
  \]
\end{proof}

\begin{remark}
  Note that the estimate in the previous Lemma can be improved
  to $\norm{c}_{\mathcal{B}^*} \le (\sqrt{m}+1)\sqrt{2p}$ if $n$ is some
  power of $m$, because in this case, with the notation of the Lemma,
  we have $\ell^+-\ell^- = p$.
  Also we have the obvious estimate $\norm{c}_{\mathcal{B}^*} \le \sqrt{p}$
  in case the family $\mathcal{B}$ contains all intervals.
\end{remark}

\subsection{Main Part of the Proof}

In order to estimate the {multiscale distance function} $d_n$,
we need to approximate $D^k f$ by a function of the
form $D^k S_n^*\omega$, where $\omega\in\R^{\Gamma_n}$
is small with respect to the dual multiresolution norm.
We will perform this approximation in two steps: First,
we will show that a spline of order $k+1$ defined on
a coarser grid than $\Gamma_n$ can be approximated
well by a function of the form $D^k S_n^*\omega$ in such a way
that the dual multiresolution norm of $\omega$ increases sufficiently 
slowly with the decreasing grid size (see Lemma~\ref{le:approx1}).
In the second step, we then approximate $D^k f$ by a spline
$g$ of order $k+1$. Balancing the grid on which $g$ is
defined with $n$, then gives us the behavior of $d_n$
claimed in Proposition~\ref{pr:distance_1d}.

\begin{lemma}\label{le:approx1}
  Let $\Gamma\subset \T$ be a finite set and $g \in \mathcal{S}_{k+1}(\Gamma;\T)$
  with $\int_{\T} g\,dx = 0$.
  Assume moreover that
  \[
  \tau_{\min} := \min\set{\dist(x,y)}{x\neq y \in \Gamma} > \frac{2k+2}{n}
  \]
  and that $\mathcal{B}$ is regular,
  denote
  \[
  \tau_{\max} := \max\set{\dist(x,y)}{(x,y) \subset \T\setminus\Gamma}
  \]
  and let $1 \le q \le \infty$.
  Then there exists $c \in \R^{\Gamma_n}$ and constants $C_1$, $C_2 > 0$
  only depending on $k$, $q$, and $\mathcal{B}$ such that
  \[
  \begin{aligned}
  \norm{g - (S_n^*c)^{(k)}}_{L^2} &\le C_1 \frac{\norm{g^{(k)}}_{L^2}}{n^{k}}\\
  \norm{c}_{\mathcal{B}^*} &\le C_2 \norm{g^{(k)}}_{L^q}\abs{\Gamma}^{1-1/q} n^{1/q-1}(n\tau_{\max})^{(1/2-1/q)_+}.
  \end{aligned}
  \]
\end{lemma}

\begin{proof}
  Let $h$ be the best approximation of $g$
  in $\Span\{\psi_{i,n}^k : i\in\Gamma_n\}$ in the $L^2$ sense, see~\eqref{eq:def:psi}. 
  Then we can write
  \[
  h = \sum_{i=0}^{n-1} \tilde{c}_i \psi_{i,n}^k
  \]
  for some coefficients $\tilde{c}_i \in \R$.
  Because the functions $\psi_{i,n}^k$ are not linearly independent,
  the coefficients $\tilde{c}_i$ are not unique. It is, however,
  possible to choose them in such a way that
  \[
  \sum_{i=0}^{n-1} \tilde{c}_i = 0.
  \]
  Then
  \begin{equation}\label{eq:approx:hcoef}
    h = \sum_{i=0}^{n-1} \tilde{c}_i\psi_{i,n}^k = \sum_{i=0}^{n-1} \tilde{c}_i Q_i^k.
  \end{equation}
  Now note that the fact that $\int_{\T} g\,dx = 0$ implies that
  $h$ is at the same time the best approximation of $g$ in $\mathcal{S}_k(\Gamma_n;\T)$.
  Thus~\eqref{eq:approx:hcoef} shows that, actually, the coefficients
  $\tilde{c}_i$ are the coefficients of the $k$-th order spline that approximates
  $g$ best in the $L^2$-sense.
  Thus it follows from Proposition~\ref{prop:spline_approx} that 
  \begin{equation}\label{eq:approx:on:fine:grid}
    \norm{h-g}_{L^2} \le C_1 \frac{\norm{g^{(k)}}_{L^2}}{n^{k}}.
  \end{equation}

  Now let 
  \[
  c_i:=(-1)^k n^{k-1} (D^k_{-}\tilde{c})_i, \quad \text{ for } i=0,\ldots,n-1,
  \]
  which implies that
  \[
  h = \sum_{i=0}^{n-1} c_i \varphi_{i,n}^{(k)} = (S_n^* c)^{(k)}.
  \]
  We will next derive an upper bound for $\norm{c}_{\mathcal{B}^*}$. 

  Since $h$ is the best approximation of $g$ within $\mathcal{S}_k(\Gamma_n;\T)$,
  it follows that
  \[
  \inner{h}{Q_j^k}_{L^2} = \inner{g}{Q_j^k}_{L^2}
  \]
  for all $j$. Applying $r$-th order finite differences
  to these vectors,
  we obtain that
  \[
  D_-^r\bigl((\inner{h}{Q_j^k}_{L^2})_j\bigr) = D_-^r\bigl((\inner{g}{Q_j^k}_{L^2})_j\bigr)
  \]
  for all $r$.
  From this, we obtain that
  \[
  \inner{h}{D_{\frac{1}{n},+}^r Q_j^k}_{L^2} = \inner{g}{D_{\frac{1}{n},+}^r Q_j^k}_{L^2}
  \]
  for all $j$. Since $(D_{\frac{1}{n},+}^r)^* = (-1)^r D_{\frac{1}{n},-}^r$,
  this further implies that
  \begin{equation}\label{eq:bestapp}
    \inner{D_{\frac{1}{n},-}^r h}{Q_j^k}_{L^2} = \inner{D_{\frac{1}{n},-}^r g}{Q_j^k}_{L^2}
  \end{equation}
  for all $j$ and all $r$. Next we note that
  \begin{equation}\label{eq:hrepresent}
  D_{\frac{1}{n},-}^{k+1} h = \sum_{i=0}^{n-1} (D_-^{k+1}\tilde{c})_i Q_i^k = (-1)^k n^{1-k} \sum_{i=0}^{n-1} (D_- c)_i Q_i^k.
  \end{equation}
  Now let $j/n \in \Gamma_n$ be such that $j/n \not\in \Gamma+(-k/n, (k+1)/n)$
  and let $x \in \supp(Q_j^k) = [j/n, (j+k)/n]$. Then the fact that $g$ is a polynomial
  of degree $k$ outside of $\Gamma$ implies that
  \[
  (D_{\frac{1}{n},-}^{k+1}g)(x) = 0.
  \]
  As a consequence, we obtain from~\eqref{eq:bestapp} with $r=k+1$
  and~\eqref{eq:hrepresent} that
  \[
  0 = \inner{D_{\frac{1}{n},-}^{k+1}g}{Q_j^k}_{L^2}
  = (-1)^k n^{1-k} \sum_{i=0}^{n-1} (D_- c)_i \inner{Q_i^k}{Q_j^k}.
  \]
  Since this holds for every $j\in \Gamma_n$ with $j/n \not\in \Gamma+(-k/n, (k+1)/n)$,
  it follows from the properties of B-splines that
  \[
  (D_- c)_j = 0
  \]
  for all $j$ such that $j/n \not\in \Gamma+(-k/n, (k+1)/n)$.

  Now denote by $I \subset \Gamma_n$ the set of all points $i/n$ for which
  $i/n \not\in \Gamma+(-k/n, (k+1)/n)$. Then the set $I$ consists
  of $\abs{\Gamma}$ disjoint sets $I_j\subset \T$, $j=1,\ldots,\abs{\Gamma}$,
  of subsequent grid points. The considerations above imply that
  for each of these sets $I_j$ there exists $\omega_j\in\R$ such that
  $c_i = \omega_j$ for $i/n \in I_j$.
  Therefore Lemma~\ref{le:cover} implies that
  \begin{equation}\label{eq:approx:cBstar}
    \norm{c}_{\mathcal{B}^*} \le \sum_{j=1}^{\abs{\Gamma}} C\abs{\omega_j}\sqrt{n(I_j,n)} + 
    \sum_{i\not\in I} \abs{c_i}
  \end{equation}
  for some constant $C > 0$ only depending on $\mathcal{B}$.
  Now define
  \[
  t_i := \begin{cases}
    1 & \text{ if } i/n \not\in I,\\
    C\dfrac{1}{\sqrt{n(I_j,n)}} & \text{ if } i/n \in I_j \text{ for some } j.
  \end{cases}
  \]
  Then the right hand side term in~\eqref{eq:approx:cBstar} can also
  be written as a sum over all products $t_i\abs{c_i}$, $i=0,\ldots,n-1$.
  Therefore
  \[
  \norm{c}_{\mathcal{B}^*} \le \sum_{i=0}^{n-1} t_i \abs{c_i}.
  \]
  Applying H\"older's inequality gives
  \begin{multline*}
    \norm{c}_{\mathcal{B}^*} \le \norm{c}_q \norm{t}_{q_*}
    = \norm{c}_q \Bigl(\abs{\Gamma_n\setminus I} 
    + C^{q_*}\sum_{j=1}^{\abs{\Gamma}} n(I_j,n)^{1-q_*/2}\Bigr)^{1/q_*}\\
    \le \norm{c}_q \Bigl(2k\abs{\Gamma}+C^{q_*} \sum_{j=1}^{\abs{\Gamma}} n(I_j,n)^{1-q_*/2}\Bigr)^{1/q_*}
  \end{multline*}
  for any $1 \le q \le \infty$ and $q_* = q/(q-1)$. 
  Since $1 \le n(I_j,n) \le n\tau_{\max}$ for all $j$, this further implies that
  \begin{multline}\label{eq:cstar}
    \norm{c}_{\mathcal{B}^*} \le \norm{c}_q\bigl(2k\abs{\Gamma} + C^{q_*}\abs{\Gamma}(n\tau_{\max})^{(1-q_*/2)_+}\bigr)^{1/q_*}\\
    \le C\norm{c}_q \abs{\Gamma}^{1/q_*}(n\tau_{\max})^{(1/q_*-1/2)_+}
    = C\norm{c}_q \abs{\Gamma}^{1-1/q}(n\tau_{\max})^{(1/2-1/q)_+}.
  \end{multline}

  Now note that~\eqref{eq:bestapp} implies that $D_{1/n,-}^k h$ is
  the best approximating spline in the $L^2$-sense of
  $D_{1/n,-}^k g$. Thus the definition of $c$ and Propositions~\ref{prop:size_coef},
  \ref{prop:Linfty_norm_L2_proj} and~\ref{prop:bnd:fd} imply that
  \begin{multline*}
  \norm{c}_q =  n^{k-1} \norm{D_-^k \tilde{c}}_q 
  \le C n^{k-1+1/q} \norm{D_{1/n,-}^k h}_{L^q} 
  \le  C n^{k-1+1/q} \norm{D_{1/n,-}^k g}_{L^q}
  \le C n^{-1+1/q} \norm{g^{(k)}}_{L^q}.
  \end{multline*}
  Together with~\eqref{eq:cstar} this shows that
  \[
  \norm{c}_{\mathcal{B}^*} \le C \norm{g^{(k)}}_{L^q}\abs{\Gamma}^{1-1/q} n^{1/q-1}(n\tau_{\max})^{(1/2-1/q)_+}
  \]
  for some constant $C > 0$.
\end{proof}

\begin{proof}[of Proposition~\ref{pr:distance_1d}]
  Assume first that $p\ge 2$.
  Proposition~\ref{prop:spline_approx} applied with
  $u = f^{(k)} \in B^{s,p'}_{p,0}(\T)$, $m=k+1$, and $q=p$
  implies for every $\lambda \in\N$ the existence of
  a spline $g \in \mathcal{S}_{k+1}(\Gamma_{\lambda};\T)$ such that
  \[
  \begin{aligned}
    \norm{f^{(k)}-g}_{L^2} &\le C \frac{\norm{f}_{B_{p,0}^{k+s,p'}}}{{\lambda}^s},\\
    \norm{g^{(k)}}_{L^p} &\le C \frac{\norm{f}_{B_{p,0}^{k+s,p'}}}{{\lambda}^{s-k}}.
  \end{aligned}
  \]
  Next we obtain from Lemma~\ref{le:approx1} the existence of a vector $c \in \R^{\Gamma_n}$
  such that
  \[
  \begin{aligned}
    \norm{g-(S_n^* c)^{(k)}}_{L^2} &\le C \frac{\norm{g^{(k)}}_{L^2}}{n^k},\\
    \norm{c}_{\mathcal{B}^*} &\le C \norm{g^{(k)}}_{L^p} {\lambda}^{1/2} n^{-1/2},
  \end{aligned}
  \]
  provided that ${\lambda}$ is sufficiently large (here we use that, in the notation of the Lemma,
  $\abs{\Gamma} = {\lambda}$ and $\tau_{\max} = 1/{\lambda}$).
  Combining these estimates, it follows that, for
  \[
  t \ge C\norm{f}_{B_{p,0}^{k+s,p'}} n^{-1/2} {\lambda}^{1/2-s+k}
  \]
  we have
  \[
  d_n(t) \le C\norm{f}_{B_{p,0}^{k+s,p'}} {\lambda}^{-s}(1+{\lambda}^{k}n^{-k}).
  \]
  Choosing
  \[
  {\lambda} \sim n^{1/(2s+2k+1)}(\log n)^{-2r/(2k+2s+1)},
  \]
  we obtain that
  \[
  \min_{t\ge 0} \bigl(d_n(t) + (\log n)^{r/2}t^{1/2}\bigr)
  = \mathcal{O}\bigl(n^{-\mu}(\log n)^{2r\mu}\bigr)
  \qquad\text{ with }\qquad \mu = \frac{s}{2s+2k+1}
  \]
  as $n\to\infty$.

  Now let $p \le 2$. Again applying Proposition~\ref{prop:spline_approx}
  with $u = f^{(k)} \in B^{s,p'}_{p,0}(\T)$ and $m=k+1$,
  but now with $q=2$ yields $g \in \mathcal{S}_{k+1}(\Gamma_{\lambda};\T)$ such that
  \[
  \begin{aligned}
    \norm{f^{(k)}-g}_{L^2} &\le C \frac{\norm{f}_{B_{p,0}^{k+s,p'}}}{{\lambda}^{s-1/p+1/2}},\\
    \norm{g^{(k)}}_{L^2} &\le C \frac{\norm{f}_{B_{p,0}^{k+s,p'}}}{{\lambda}^{s-k-1/p+1/2}},
  \end{aligned}
  \]
  and we obtain, for ${\lambda}$ sufficiently large,
  from Lemma~\ref{le:approx1} the existence of $c\in\R^{\Gamma_n}$
  with
  \[
  \begin{aligned}
    \norm{g-(S_n^* c)^{(k)}}_{L^2} &\le C \frac{\norm{g^{(k)}}_{L^2}}{n^k},\\
    \norm{c}_{\mathcal{B}^*} &\le C \norm{g^{(k)}}_{L^2} {\lambda}^{1/2} n^{-1/2}.
  \end{aligned}
  \]
  This shows that, for
  \[
  t \ge C\norm{f}_{B_{p,0}^{k+s,p'}} n^{-1/2} {\lambda}^{1/p-s+k},
  \]
  we have
  \[
  d_n(t) \le C\norm{f}_{B_{p,0}^{k+s,p'}} {\lambda}^{1/p-1/2-s}(1+{\lambda}^{k}n^{-k}).
  \]
  Choosing
  \[
  {\lambda} \sim n^{1/(2k+2s+2-2/p)}(\log n)^{-2r/(2k+2k+s-2/p)},
  \]
  we obtain that
  \[
  \min_{t\ge 0} \bigl(d_n(t)+(\log n)^{r/2}t^{1/2}\bigr) =\mathcal{O}\bigl(n^{-\mu}(\log n)^{2r\mu}\bigr)
  \qquad\text{ with }\qquad
  \mu = \frac{s-1/p+1/2}{2(s+k+1-1/p)},
  \]
  which proves the assertion.
  
 {The above argument holds also for $f \in W^{k+s, p}_{0}(\T)$ if we replace $\norm{\cdot}_{B^{k+s, p'}_{p,0}}$ by $\norm{\cdot}_{W^{k+s, p}_{0}}$. }
\end{proof}

\section{Proof of Theorem~\ref{th:oversmoothing}}\label{sec:over:smooth}

\begin{lemma}\label{lem:continuousMRnorm}
Let $\mathcal{B}$ be a family of intervals, $f \in \mathcal{C}^1(\T)$, and $F(t) := \int_{0}^t f(x)dx$. Then,
\[
\frac{\norm{S_n f}_{\mathcal{B}}}{\sqrt{n}} \le \norm{F}_{W^{1/2,\infty}}+\frac{\norm{Df}_{L^{\infty}}}{2n}.
\] 
\end{lemma}
\begin{proof}
Denote in the following
\[
\mathcal{B}_n := \set[B]{[i/n,j/n]}{\text{ there exists }
B \in \mathcal{B} \text{ such that } B\cap \Gamma_n = \{i/n,\ldots,j/n\}}.
\]
Then
\begin{align*}
\frac{\norm{S_n f}_{\mathcal{B}}}{\sqrt{n}} = &\max_{[i/n, j/n] \in \mathcal{B}_n} \frac{\sqrt{n}}{\sqrt{j-i+1}}\Bigl|{\frac{1}{n}\sum_{k = i}^j f\bigl(\frac{k}{n}\bigr)}\Bigr|\\
\le& \max_{[i/n,j/n] \in \mathcal{B}_n} \biggl\{ \frac{\sqrt{n}}{\sqrt{j-i+1}} \Bigl|{\int_{i/n}^{(j+1)/n}f(x)dx}\Bigr|  \\
& {} \qquad \qquad \quad + \frac{\sqrt{n}}{\sqrt{j-i+1}} \sum_{k = i}^j \Bigl|\frac{1}{n}f\bigl(\frac{k}{n}\bigr) - \int_{k/n}^{(k+1)/n}f(x)dx\Bigr| \biggr\} \\
\le &\sup_{s\neq t \in \T}\frac{\abs{F(t)-F(s)}}{\sqrt{t-s}}  \\
& {} \qquad \qquad \quad + \max_{[i/n, j/n] \in \mathcal{B}_n} \frac{\sqrt{n}}{\sqrt{j-i+1}}\sum_{k = i}^j \int_{k/n}^{(k+1)/n}\Bigl|f\bigl(\frac{k}{n}\bigr)-f(x)\Bigr|dx \\ 
\le &\norm{F}_{W^{1/2, \infty}} + \max_{[i/n, j/n] \in \mathcal{B}_n}\frac{\sqrt{n}}{\sqrt{j-i+1}}\frac{j-i+1}{2n^2}\norm{Df}_{L^\infty} \\
\le &\norm{F}_{W^{1/2,\infty}}+\frac{\norm{Df}_{L^{\infty}}}{2n}.
\end{align*}
\end{proof}

\begin{proof}[of Theorem~\ref{th:oversmoothing}] 
Let $\tilde{\epsilon} > 0$ be fixed and set 
$$
{\lambda}:=\left\lfloor\left(\frac{n}{\log n}\right)^{\frac{1}{2s+1}}(\log n)^{\tilde{\epsilon}}\right\rfloor.
$$ 
Let {$G_{\lambda}(t) \in \mathcal{S}_{k+2}(\Gamma_{\lambda};\T)$} be the approximation spline of $F(t):=\int_{0}^t f(x)dx$ as in Proposition~\ref{prop:spline_approx}, and $g_{\lambda}(t):= G_{\lambda}'(t).$ It follows that $g_{\lambda} \in H^k_0(\T)$, and 
\begin{align*}
\norm{f-g_{\lambda}}_{L^q}& = \norm{F' - G_{\lambda}'}_{L^q} \le C \frac{\norm{F}_{W^{s+1,\infty}}}{{\lambda}^s} \\
& \le C\left(\frac{\log n}{n}\right)^{\frac{s}{2s+1}}(\log n)^{-s\tilde{\epsilon}}\norm{f}_{W^{s,\infty}}, \\
\norm{D^kg_{\lambda}}_{L^2}& = \norm{D^{k+1}G_{\lambda}}_{L^2} \le C {\lambda}^{k-s}\norm{F}_{W^{s+1,\infty}} \\
& \le   C\left(\frac{n}{\log n}\right)^{\frac{k-s}{2s+1}}(\log n)^{(k-s)\tilde{\epsilon}}\norm{f}_{W^{s,\infty}}. 
\end{align*}
The second relation implies that
\begin{multline}\label{eq:est_dN_tilde}
\tilde{d}_n(t):=\min_{\norm{w}_{\mathcal{B}^*} \le t} \norm{D^k S_n^*w - 2D^k g_{\lambda}}_{L^2} \le \tilde{d}_n(0) = 2\norm{D^kg_{\lambda}}_{L^2} \\
 \le 2C\left(\frac{n}{\log n}\right)^{\frac{k-s}{2s+1}}(\log n)^{(k-s)\tilde{\epsilon}}\norm{f}_{W^{s,\infty}}
\end{multline}
for every $t \ge 0$.
By Proposition~\ref{prop:spline_approx} and Lemma~\ref{lem:continuousMRnorm}, we have
\begin{align*}
\norm{S_n f - S_n g_{\lambda}}_{\mathcal{B}} & \le \sqrt{n}\norm{F-G_{\lambda}}_{W^{1/2, \infty}}+o(1) \\
& \le C \sqrt{n}\frac{\norm{f}_{W^{s,\infty}}}{{\lambda}^{s+1/2}} \le C {(\log n)^{\frac{1}{2}-(s+\frac{1}{2})\tilde{\epsilon}}}\norm{f}_{W^{s,\infty}}.
\end{align*}
for sufficiently large $n$. Consequently 
\begin{equation}\label{eq:estMRnorm}
\begin{aligned}
\norm{S_n g_{\lambda} - y_n}_{\mathcal{B}} & \le \norm{S_n g_{\lambda} - S_n f}_{\mathcal{B}} + \norm{S_n f - y_n}_{\mathcal{B}} \\
& \le C {(\log n)^{\frac{1}{2}-(s+\frac{1}{2})\tilde{\epsilon}}}\norm{f}_{W^{s,\infty}} + \norm{\xi_n}_{\mathcal{B}}.
\end{aligned}
\end{equation}
Set $\tilde{\gamma}_n := C_0 \sqrt{\log n} < \gamma_n$ with some $C_0 > \sigma\sqrt{5+2k/d}$. Then $\norm{\xi_n}_{\mathcal{B}} \le \tilde{\gamma}_n$, together with~\eqref{eq:estMRnorm}, implies that  $\norm{S_ng_{\lambda} - y_n}_{\mathcal{B}} \le \gamma_n$ for large enough $n$. In such case we can apply Theorem~\ref{th:constrainedrate}, but with $f$ replaced by its approximation $g_{\lambda}$, and obtain the estimate  
\begin{multline*}
\norm{\hat{f}_{\gamma_n}-g_{\lambda}}_{L^q} \le C\max\biggl\{\frac{\gamma_n^{2\vartheta}}{n^\vartheta}\min_{t \ge 0} \left(\tilde{d}_n(t)+(\gamma_n t)^{1/2}\right)^{1-2\vartheta}, 
 \frac{\gamma_n}{n^{1/2}}, 
\frac{1}{n^{\vartheta'}}\min_{t \ge 0} \left(\tilde{d}_n(t) + (\gamma_n t)^{1/2}\right) \biggr\}, 
\end{multline*}
with $\vartheta = k/(2k+1)$, $\vartheta' = 2k^2/(2k+1)$. By~\eqref{eq:est_dN_tilde}, this further implies that,
for sufficiently large $n$, the estimate 
\begin{align*}
&\norm{\hat{f}_{\gamma_n}-g_{\lambda}}_{L^q} \\
\le & C \max\left\{ \frac{(\log n)^{\frac{s}{2s+1}+\epsilon}}{n^{\frac{s}{2s+1}}} \norm{f}_{W^{s,\infty}}^{\frac{1}{2k+1}},  \frac{(\log n)^r}{n^{1/2}},  \frac{(\log n)^{(k-s)\tilde{\epsilon}-\frac{k-s}{2s+1}}}{n^{\frac{s}{2s+1}+\vartheta\frac{4ks-1}{2s+1}}}\norm{f}_{W^{s,\infty}}\right\} \\
\le & C{(\log n)^{\frac{s}{2s+1}+\epsilon}}{n^{-\frac{s}{2s+1}}} \norm{f}_{W^{s,\infty}}^{\frac{1}{2k+1}},
\end{align*}
with $\epsilon = \frac{(k-s)\tilde{\epsilon}+(2r-1)k}{2k+1} >
\frac{(2r-1)k}{2k+1}$. Note that $\lim_{n \to \infty} \Prob{\norm{\xi_n}_{\mathcal{B}} > \tilde{\gamma}_n} = 0$ by Proposition~\ref{prop:mulitresolution_probability}. 
Thus we obtain that
\begin{align*}
\norm{\hat{f}_{\gamma_n}-f}_{L^q} 
& \le\norm{\hat{f}_{\gamma_n}-g_{\lambda}}_{L^q} + \norm{g_{\lambda}-f}_{L^q} \\
& \le  C{(\log n)^{\frac{s}{2s+1}+\epsilon}}{n^{-\frac{s}{2s+1}}} \norm{f}_{W^{s,\infty}}^{\frac{1}{2k+1}} + C {(\log n)^{\frac{s}{2s+1}-s\tilde{\epsilon}}}{n^{-\frac{s}{2s+1}}} \norm{f}_{W^{s,\infty}} \\
& \le C {(\log n)^{\frac{s}{2s+1}+{\epsilon}}}{n^{-\frac{s}{2s+1}}}\norm{f}_{W^{s,\infty}}^{\frac{1}{2k+1}}
\end{align*}
almost surely as $n \to \infty$. 

If $p(t) := \Prob{\norm{\xi_n}_{\mathcal{B}} \le t}$, it follows that for $n$ sufficiently large,
\begin{align*}
&\E{\norm{\hat{f}_{\gamma_n}-f}_{L^q}} \\
\le &\Prob{\norm{\xi_n}_{\mathcal{B}} \le \tilde{\gamma}_n} C  {(\log n)^{\frac{s}{2s+1}+{\epsilon}}}{n^{-\frac{s}{2s+1}}}\norm{f}_{W^{s,\infty}}^{\frac{1}{2k+1}} \\& \qquad\qquad\qquad
+\int_{\tilde{\gamma}_n}^{\infty}\sup\left\{\norm{\hat{f}_{\gamma_n}-f}_{L^\infty}:\norm{\xi_n}_{\mathcal{B}} = t\right\}dp(t)\\
\le & C {(\log n)^{\frac{s}{2s+1}+{\epsilon}}}{n^{-\frac{s}{2s+1}}}\norm{f}_{W^{s,\infty}}^{\frac{1}{2k+1}} \\& \qquad\qquad\qquad
+ \int_{\tilde{\gamma}_n}^{\infty}\sup\left\{\norm{\hat{f}_{\gamma_n}-f}_{L^\infty}:\norm{\xi_n}_{\mathcal{B}} = t\right\}dp(t).
\end{align*}
As in the proof of Theorem~\ref{th:rate_expect}, we see that
\[
\int_{\tilde{\gamma}_n}^{\infty}\sup\left\{\norm{\hat{f}_{\gamma_n}-f}_{L^\infty}:\norm{\xi_n}_{\mathcal{B}} = t\right\}dp(t) \le C n^{\frac{2k+5d}{2d}} \tilde{\gamma}_n \exp\left(-\frac{\tilde{\gamma}_n^2}{2\sigma^2}\right) \to 0, 
\] 
as  $n \to \infty. $ 
\newline
It is easy to see that the above argument also holds for $f \in B^{s, p'}_{\infty,0}(\T)$ with $1\le p' \le \infty$. This completes the proof. 
\end{proof}

\bibliographystyle{apalike}

\end{document}